\documentclass[12pt]{report} 
\usepackage{utdiss1-18}  
\usepackage{amsmath,amsthm,amsfonts,amscd}
\usepackage{eucal}
\usepackage{verbatim}     
\usepackage{makeidx}      
\usepackage{psfig}        
\usepackage{amscd}\usepackage{amsmath}\usepackage{amssymb}
\usepackage{amsthm}\usepackage{latexsym}\usepackage{verbatim}
\usepackage{epsfig}\usepackage{graphicx}
\usepackage{pstricks-add}
\newtheorem{thm}{Theorem}[chapter]
\newtheorem{cor}[thm]{Corollary}

\newtheorem{lemma}[thm]{Lemma}
\newtheorem{prop}[thm]{Proposition}

\newtheorem{notation}[thm]{Notation}
\newtheorem{example}[thm]{Example}

\theoremstyle{remark}

\newtheorem{remark}[thm]{Remark}
\newtheorem{note}[thm]{Note}
\newtheorem{questions}[thm]{Questions}
\theoremstyle{definition}
\newtheorem{dfn}[thm]{Definition}

\newcommand{\C}{\mathbb{C}}
\newcommand{\M}{\mathcal{M}}\newcommand{\N}{\mathbb{N}}
\newcommand{\p}{\mathbb{P}}\newcommand{\Q}{\mathbb{Q}}
\newcommand{\R}{\mathbb{R}}\newcommand{\s}{\mathbb{S}}
\newcommand{\T}{\mathbb{T}}\newcommand{\Z}{\mathbb{Z}}
\newcommand{\LL}{\mathcal{LL}}
\renewcommand{\H}{\mathcal{H}}
\renewcommand{\O}{\mathcal{O}}
\newcommand{\ol}{\overline}\newcommand{\tr}{\mbox{tr}}
\newcommand{\ind}{\mbox{ind}}

\author{Brian Paul Katz}                           
\title{Tropical Hurwitz Spaces}                
\address{1129 38$^{th}$ St\\ Rock Island, IL 61201}  
\supervisor{David Helm}
\member[Eric Katz][Daniel Allcock][David Ben-Zvi][Sean Keel]{Brendan Hassett}

\previousdegrees{B.A.}

\newcommand{\latexe}{{\LaTeX\kern.125em2%
                      \lower.5ex\hbox{$\varepsilon$}}}

\chardef\bslchar=`\\ 

\makeatletter
\def\square{\RIfM@\bgroup\else$\bgroup\aftergroup$\fi
  \vcenter{\hrule\hbox{\vrule\@height.6em\kern.6em\vrule}%
                                              \hrule}\egrou0}
\makeatother

\makeindex    

\begin{document}

\copyrightpage                   

\signaturepage                   

\titlepage                              

\begin{dedication}              
\index{Dedication@\emph{Dedication}}%
I would like to dedicate this thesis to my dear friends, without whom this would not have been possible.
\end{dedication}

\begin{acknowledgments}
\index{Acknowledgments@\emph{Acknowledgments}}%
I would like to thank David Helm, Michael Starbird, and Eric Katz for many helpful conversations and their generous application of time, patience, and encouragement. I especially thank Eric for the origins of many of the ideas below, David for the guidance to make this document thorough and rigorous, and Mike for determination to see it to completion.
\end{acknowledgments}

\utabstract       
\index{Abstract}
                  
Hurwitz numbers are a weighted count of degree $d$ ramified covers of curves with specified ramification profiles at marked points on the codomain curve. Isomorphism classes of these covers can be included as a dense open set in a moduli space, called a Hurwitz space. The Hurwitz space has a forgetful morphism to the moduli space of marked, stable curves, and the degree of this morphism encodes the Hurwitz numbers.

Mikhalkin has constructed a moduli space of tropical marked, stable curves, and this space is a tropical variety. In this paper, I construct a tropical analogue of the Hurwitz space in the sense that it is a connected, polyhedral complex with a morphism to the tropical moduli space of curves such that the degree of the morphism encodes the Hurwitz numbers.

\tableofcontents       



\chapter{Introduction}
\textit{This document combines Hurwitz numbers from classical enumerative geometry and the moduli space of curves from tropical geometry. This introduction establishes the basic definitions in each domain and frames them for the work to come. Then it states the original results and motivates the tools that will be used to prove them.}

\section{Hurwitz Numbers}
\textit{In this section, we define Hurwitz numbers and show how they can be computed in the class algebra.}

\subsection{Counting Ramified Covers}
\textit{In this subsection, we define Hurwitz numbers in enumerative geometry. For a fixed ramification profile, the Hurwitz number will be a weighted count of ramified covers with that profile.}

\begin{dfn}\cite{GJ}
Fix $d \in \N$. Then a \textbf{ramified cover of degree $d$} is a morphism $f: D \to C$, where $C$ and $D$ are compact curves over $\C$, such that $|f^{-1}(Q)| = d$ for all but a finite number of points $Q$ in $C$. The points in $C$ that do not have $d$ preimages are called \textbf{ramification values} or \textbf{branch points} of $f$. Two ramified covers $f: D \to C$ and $f': D' \to C'$ are isomorphic if there are isomorphisms $d: D \to D'$ and $c: C \to C'$ such that $f' \circ d = c \circ f$.
\end{dfn}

\begin{remark}
We will require that $C$ be connected below, but $D$ need not be. 
\end{remark}

\begin{lemma}
Let $f: D \to C$ be a ramified cover of degree $d$ and $P \in D$. Then there exists $m_P \in \N$ such that $f$ is locally isomorphic to the map $z \mapsto z^{m_P}$. The integer $m_P$ is called the \textbf{ramification index} of $f$ at $P$.
\end{lemma}

\begin{proof}
This is proved by Proposition IV.2.2 of Hartshorne, \cite{Hartshorne}.
\end{proof}

The given definition of a ramified cover requires that any point that is \textit{not} a branch point has $d$ distinct preimages. The next lemma uses the ramification indices to extend this count to \textit{all} values in the codomain, $C$.

\begin{lemma}\label{ramificationpartition}
Let $f: D \to C$ be a ramified cover of degree $d$ and fix $Q \in C$. Then
$$\sum_{P \in f^{-1}(Q)} m_P = d.$$
\end{lemma}

\begin{proof}
This is proved by Proposition II.6.9 of Hartshorne, \cite{Hartshorne}.
\end{proof}

\begin{dfn}
Let $f: D \to C$ be a ramified cover of degree $d$ and $Q \in C$. Then, by lemma \ref{ramificationpartition}, the ramification indices at points $P \in f^{-1}(Q)$ form an integer partition of $d$. This integer partition will be denoted by $\sigma(Q)$.
\end{dfn}

\begin{notation}
Let $\sigma$ be an integer partition of $d$. Let $n_i$ be the number of parts of this integer partition of size $i$. Then we can represent this data as $\sigma = (1^{n_1} 2^{n_2} \cdots d^{n_d})$. There is no information lost by dropping those $i$ with $n_i = 0$ from this notation.
\end{notation}

\begin{example}\label{trivialram}
Let $f : D \to C$ be a ramified cover of degree $d$. By definition, all but a finite number of points in $C$ have $d$ distinct preimages. Let $Q$ be one of the points with $d$ distinct preimages, and let $\{\hat{Q}_1,\ldots,\hat{Q}_d\}$ be those preimages. Lemma \ref{ramificationpartition} implies that the ramification indices are all positive integers summing to $d$, so $m_{\hat{Q}_i} = 1$ for all $i$. Thus $\sigma(Q) = (1^d)$ for all but a finite number of points $Q \in C$.
\end{example}

\begin{dfn}\cite{CJM}
Let $d$ be a natural number and fix $n$ points, $Q_1,\ldots,Q_n$, on $\p^1$ and $\overline{\sigma} = \{\sigma_1, \ldots, \sigma_n\}$ a collection of integer partitions of $d$. Then a \textbf{Hurwitz number}, $h(\overline{\sigma})$, is defined as a weighted count of (isomorphism classes of)
\begin{center}
degree $d$ ramified covers, $f: D \to \p^1$ such that:
\begin{itemize}
\item $D$ is a smooth curve;
\item $f$ is unramified over $\p^1 \setminus \{Q_1,\ldots,Q_n\}$; and
\item $f$ ramifies with profile $\sigma(Q_i) = \sigma_i$.
\end{itemize}
Each cover $f$ is counted with weight $\frac{1}{|Aut(f)|}$.
\end{center}
\end{dfn}

The definition above of Hurwitz numbers is a little different from the usual one. Traditionally, the genus of $D$ is specified and included in the notation. This allows for a partial collection of ramification profiles to be specified, while an unknown number of points with profile $(1^{d-2}2^1)$ are left unspecified. If you know the ramification profiles, then you can read off the degree of the covers. Moreover, from the ramification profiles, you can compute the genus, $g$, using the following famous result. The two definitions are the same except for some notation. We will assume all ramification is listed in the profiles, so we gave that version of the definition.

\begin{dfn}
Let $\sigma = (1^{n_1} 2^{n_2} \cdots d^{n_d})$ be an integer partition of $d$. Define $\ell(\sigma) = \sum_i n_i$, the number of parts in the integer partition $\sigma$. Also define $r(\sigma) = d - \ell(\sigma)$.
\end{dfn}

\begin{note}
By Example \ref{trivialram}, for all but a finite number of points $Q \in C$, $r(\sigma(Q)) = r(1^d) =  d - \ell(1^d) = d-d = 0$.
\end{note}

\begin{thm}[Riemann-Hurwitz Formula]
Let $f: D \to C$ be a ramified cover of degree $d$. Let $g(C)$ and $g(D)$ be the genera of these curves. Then
$$2 - 2g(D) = d(2-2g(C)) - \sum_{Q \in C} r(\sigma(Q)).$$
\end{thm}

\begin{proof}
This is proved by Corollary 2.4 of Hartshorne, \cite{Hartshorne}.
\end{proof}

\subsection{Counting Covering Spaces}
\textit{Hurwitz numbers count ramified covers. Instead, they can be interpreted as counting related covering space maps, using the Riemann Extension Theorem.}

\begin{lemma}
Let $f : D \to \p^1$ be a ramified cover of degree $d$, and let $Q_1, \ldots, Q_n$ be the ramification values of $f$. Then 
$$f : D \setminus f^{-1}(\{Q_1, \ldots, Q_n\}) \to \p^1 \setminus \{Q_1, \ldots, Q_n\}$$
is a covering space map.
\end{lemma}

The Riemann Extension Theorem says that a covering space of a punctured copy of $\p^1$ can be completed to a ramified cover in a unique manner. See \cite{Miranda} by Miranda for a general reference. In the language of Hartshorne I.6 (\cite{Hartshorne}), the morphism $f$ induces a map of the function fields of the codomain and domain curves. Then this map of function fields can be realized as a morphism of smooth, projective curves. The curves in this new version of the morphism are each the unique smooth curve in the birational equivalence class of the original domain and codomain respectively, and the original domain and codomain curves live as (Zariski) open sets in these projective curves. Since curves are dimension $1$, the complements of the original curves are dimension $0$, meaning finite collections of points, as desired.

The automorphisms of the covering space and the ramified cover are the same, coming from the fundamental group of the punctured $\p^1$. So, counting these covering spaces, weighted by their automorphisms, is identical to computing Hurwitz numbers.

\subsection{Counting Monodromy Representations}\label{monodromy}
\textit{Counts of covering maps can be interpreted as counting monodromy representations.}

\begin{notation}
Let $f: D' \to \p^1 \setminus \{Q_1, \ldots, Q_n\}$ be a degree $d$ covering space. Fix a point $x \in \p^1 \setminus \{Q_1, \ldots, Q_n\}$, and fix a labeling of the points in $f^{-1}(x)$ as $\{\tilde{x}_1, \tilde{x}_2, \ldots, \tilde{x}_d\}$. Let $\ell$ be an element of $\pi_1(\p^1 \setminus \{Q_1, \ldots, Q_n\}, x)$, and let $\tilde{\ell}_k$ be the lift of $\ell$ with $\tilde{\ell}_k(0) = \tilde{x}_k$, the starting point of this lift path.
\end{notation}

\begin{dfn}
The \textbf{monodromy representation} of the covering space $f : D' \to \p^1 \setminus \{Q_1, \ldots, Q_n\}$ with basepoint $x$ is the homomorphism
$$\hat{f} : \pi_1(\p^1 \setminus \{Q_1, \ldots, Q_n\}, x) \to S_d$$
satisfying $\tilde{x}_{\hat{f}(\ell)(k)} = \tilde{\ell}_k(1)$, the endpoint of this lift path.
\end{dfn}

\begin{remark}
The previous definition is not clear geometrically. In short, each lift of a loop from the fundamental group gives an oriented path from $\tilde{x}_i$ to $\tilde{x}_j$ for every $i$. This map $i \to j$ specifies a permutation of the $d$ preimages. Alternately, this is the permutation representation on the cosets of $f_*\pi_1(D',\tilde{x})$ in $\pi_1(\p^1 \setminus \{Q_1, \ldots, Q_n\}, x)$.
\end{remark}

By standard covering space results, the covering space can be recovered from the map $\hat{f}$. Notice that any reordering of $\{\tilde{x}_1, \tilde{x}_2, \ldots, \tilde{x}_d\}$ produces the same covering space, so we have over-counted by a factor of $d!$.

\begin{notation}
Fix $\alpha \in S_d$. Then $\alpha$ can be written uniquely as a collection of disjoint cycles so that all of $\{1,\ldots,d\}$ appears in a cycle (up to reordering). An $i$-cycle is a cycle containing exactly $i$ elements from $\{1,\ldots,d\}$. Let $n_i$ be the number of $i$-cycles in this representation.

Notice that the sum of the cycle lengths of $\alpha$ is always $d$, so those lengths form an integer partition of $d$. The \textbf{cycle-type} of the permutation $\alpha$ is the integer partition $\sigma(\alpha) = (1^{n_1}2^{n_2}\cdots d^{n_d})$.
\end{notation}

\begin{dfn}
Let $f: D' \to \p^1 \setminus \{Q_1, \ldots, Q_n\}$ be a covering space and let $g_i$ be the loop in $\pi_1(\p^1 \setminus \{Q_1, \ldots, Q_n\}, x)$ that separates the puncture at $Q_i$ from all of the other punctures such that the puncture at $Q_i$ is on the left-hand side of the loop.
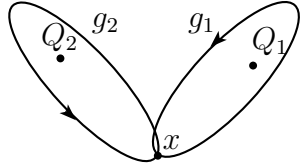
\begin{center}
\psset{xunit=0.8cm,yunit=0.8cm,algebraic=true,dotstyle=*,dotsize=3pt 0,linewidth=0.8pt,arrowsize=3pt 2,arrowinset=0.25}
\begin{pspicture*}(-1.3,-0.8)(4,2.2)
\rput{-132.88}(2.43,0.69){\psellipse(0,0)(1.7,0.56)}
\rput{132.87}(0.07,0.65){\psellipse(0,0)(1.75,0.54)}
\psline{->}(-0.23,0.17)(-0.06,0.01)
\psline{->}(2.32,1.37)(2.17,1.23)
\psdots(1.3,-0.58)
\rput[bl](1.38,-0.46){$x$}
\psdots(2.88,0.92)
\rput[bl](2.86,1.09){$Q_1$}
\rput[bl](1.8,1.46){$g_1$}
\psdots(-0.32,1.04)
\rput[bl](-0.64,1.16){$Q_2$}
\rput[bl](0.2,1.46){$g_2$}
\end{pspicture*}
\end{center}
\end{dfn}

\begin{notation}
Notice that $g_i$ does not have to act transitively on the lifts of $x$. However, the orbits of this action do partition the set of preimages; the sizes of the sets in this (set) partition form an integer partition of $d$. Write $\sigma(Q_i)$ for the integer partition of the action of $g_i$.
\end{notation}

\begin{remark}
The partition $\sigma(Q_i)$ must be the integer partition generated by the cycle-type of the image of $g_i$ in the monodromy representation: $\sigma(Q_i) = \sigma(\hat{f}(g_i))$. In addition, because each $g_i$ goes around one puncture, their product goes around all of them and hence is trivial in the fundamental group of $\p^1$, the sphere.
\end{remark}


\begin{prop}\label{pigens}
Giving $\hat{f}: \pi_1(\p^1 \setminus \{Q_1,\ldots,Q_n\}) \to S_d$ up to conjugacy is equivalent to giving a choice of generators, $\gamma_i$ for $\{\gamma_i | 1 \leq i \leq r\}$ with the single relation $\gamma_1 * \gamma_2 * \cdots * \gamma_n = 1$ such that $\sigma(\gamma_i) = \sigma(Q_i)$. In other words,
$$h(\overline{\sigma}) = \frac{|Hom^{\overline{\sigma}}(\pi_1(\p^1 \setminus \{Q_1,\ldots,Q_n\}), S_d)|}{d!}.$$
\end{prop}



\begin{remark}
There may not always be $d!$ elements in each conjugacy class, but the stabilizer is in bijection with the automorphisms of the covers, giving this simple description. The notation above comes to me from a talk by R. Cavalieri, \cite{CavTalk}.
\end{remark}

\subsection{Computing in the Class Algebra}
\textit{Counts of monodromy representations can be computed from coefficients in expressions in the class algebra. The computation uses the trace on the class algebra.}

\begin{notation}
Fix $d \in \N$ and consider the group ring $\R[S_d]$. Then every element $r \in \R[S_d]$ can be written as an $\R$-linear sum of the elements of $S_d$,
$$r = \sum_{g \in S_d} c_g \cdot g$$
where each $c_g \in \R$.
\end{notation}

\begin{remark}
The field $\R$ could be replaced with $\Q$ or even $\Z$ for our purposes below. We originally chose values in $\R$ because of the possibility of interpreting the coefficients as edge lengths.
\end{remark}

\begin{dfn}
The \textbf{class algebra} is the center of $\R[S_d]$, written $Z(\R[S_d])$.
\end{dfn}

Note that two permutations in $S_d$ have the same cycle-type if and only if they are conjugates. So the set of conjugacy classes can be identified with the set of integer partitions of $d$.

\begin{dfn}\label{basisdfn}
For a fixed integer partition $\sigma$ of $d$, let 
$$K_\sigma = \sum_{\sigma(g) = \sigma} 1 \cdot g.$$
Through a slight abuse of notation, we can think of each element of $S_d$ as an element of $\R[S_d]$. In this parlance, $K_\sigma$ is the sum of the permutations with cycle-type $\sigma$.
\end{dfn}

\begin{lemma}\label{basis}
The class algebra, $Z(\R[S_d])$, has as a (vector space) basis $$\{K_\sigma \mid \sigma \mbox{ an integer partition of } d\}.$$
\end{lemma}

\begin{proof}
First, we will show that $K_\sigma$ is in the center of the group algebra, $Z(\R[S_d])$. Consider an element $g$, which is a single permutation inside $\R[S_d]$. Then
$$g K_\sigma = g(\sum_{\sigma(\alpha) = \sigma} \alpha) = \sum_{\sigma(\alpha) = \sigma} g \alpha  = \sum_{\sigma(\alpha) = \sigma} g \alpha g^{-1}g \overset{*}{=} \sum_{\sigma(\alpha') = \sigma} \alpha' g = K_\sigma g.$$
The starred equality holds because two elements in $S_d$ are conjugate if and only if they have the same cycle-type (which is to say integer partition of $d$) and conjugation is a bijection. Any element $r \in \R[S_d]$ is an $\R$-linear combination of permutations, so this computation actually checks that $K_\sigma$ is central.

Conversely, let $z = \sum_{g\in S_d} c_g \cdot g$ be some central element in $\R[S_d]$ and suppose there is a permutation $\alpha$ so that $c_\alpha \neq 0$. Then, for any $\beta = g \alpha g^{-1}$, the coefficient of $\beta$ in $g z g^{-1}$ is $c_\alpha$. However, because $z$ is central, this coefficient is also $c_\beta$. Hence the coefficients are constant on conjugacy classes and $z$ is in the span of the $K_\alpha$. So the set $\{K_\sigma \mid \sigma \mbox{ an integer partition of } d\}$ forms a basis for the center of the group algebra as a vector space.
\end{proof}

\begin{notation}
We will write $|K_\sigma|$ for the size of the conjugacy class with integer partition $\sigma$. If $\sigma = (1^{n_1}2^{n_2} \cdots d^{n_d})$, then standard combinatorial techniques show that
$$|K_\sigma| = \frac{d!}{1^{n_1}n_1!2^{n_2}n_2! \cdots d^{n_d}n_d!}.$$
\end{notation}

\begin{dfn}
Define the function $\tr: Z(\R[S_d]) \to \R$ by
$$\tr(\sum_{g \in S_d} c_g \cdot g) = c_e.$$
\end{dfn}

\begin{example}\label{conjsize}
Let $\sigma \neq \sigma'$ be integer partitions of $d$. Then $\tr(K_\sigma K_{\sigma'}) = 0$. This is because the only way to write the identity as the product of two permutations is as the product of inverses, and in $S_d$, inverses have the same cycle-type. Similarly, $\tr(K_\sigma K_\sigma) = |K_\sigma|$ because each element has a unique inverse in its conjugacy class.
\end{example}

\begin{lemma}
The function $\tr$ is linear.
\end{lemma}

\begin{proof}
The trace function is the projection onto one of the basis elements.
\end{proof}

\begin{prop}
The number of conjugacy classes of monodromy representations, $\hat{f}$, is equal to $\frac{1}{d!} \tr(K_{\sigma_1} \ldots K_{\sigma_n})$.
\end{prop}

\begin{proof}
The trace function reads off the coefficient of the identity. So\\ $\tr(K_{\sigma_1} \ldots K_{\sigma_n})$ counts the ways to write the identity as a product of $n$ permutations with cycle-types $\overline{\sigma} = \{\sigma_1, \ldots, \sigma_n\}$. This is exactly the same count as the number of ways to choose generators for the the monodromy representation.
\end{proof}

\begin{note}
The loops $g_i$ around the punctures in the covering space give the trivial relation in any order, so it's not surprising that we end up in the class algebra, the \textit{center} of a group ring.
\end{note}

\section{Tropical Geometry}
Algebraic geometry is the study of the geometry of a set through properties of its set of regular functions. Classically, algebraic geometers studied sets that can be described locally as the simultaneous zero set of a collection of polynomials; in this case, the set of regular functions is described as a quotient of a polynomial ring over some field.

The field over which the polynomials above are defined is critical. For example, we might study the solutions to the equation $x^2 + y^2 = 1$. You will immediately note that the set of solutions to the equation depends on the allowable values for $x$ and $y$. Number theorists are interested in this set if $x$ and $y$ are restricted to being in an algebraic extension of $\Q$; a precalculus student might be interested in this set when $x$ and $y$ are real numbers. Most algebraic geometers consider the sets defined by algebraic or regular functions that are defined over an algebraically closed field like $\C$. Algebraic geometry has a very different character over different fields.

\subsection{The Tropical Numbers, $\T$}
Tropical algebraic geometry is algebraic geometry over the tropical numbers. To do tropical geometry, we must first define the tropical numbers and then specify our collections of regular functions.

\begin{dfn}
Let $\T = \R \cup \{-\infty\}$ be the \textbf{tropical numbers}.
\end{dfn}

Although the tropical numbers contain $\R$, we will give them a very different algebraic structure. We will extend the operations of $\max$ and $+$ from $\R$ to all of $\T$ to define a pair of distributing binary operations, making the tropical numbers a semi-field. 

\begin{dfn}\label{troperations}
Let $a,b \in \R$. Then we define $a \oplus b = \max\{a,b\}$ and $a \odot b = a + b$ where this maximum and addition are computed using the traditional Archimedean ordering and additive structure on $\R$. Extend the operations $\oplus$ and $\odot$ to $\T$ by letting
\begin{itemize}
\item $-\infty \oplus a = a = a \oplus -\infty$; \hspace{.1in} $-\infty \oplus -\infty = -\infty$, and
\item $-\infty \odot a = -\infty = a \odot -\infty$; \hspace{.1in} $-\infty \odot -\infty = -\infty$.
\end{itemize}
\end{dfn}

We have extended these operations to all of $\T$ in the most naive manner. We are essentially thinking of $-\infty$ as the most negative `real' number; this makes it clear that $-\infty$ only affects a maximum if all terms are $-\infty$ and it dominates any traditional sum. In the following two expressions, the right-hand expression is an abuse of notation; we will allow this abuse because the common interpretations of these expressions are consistent with the more precise definitions above.
$$-\infty \oplus 1 = 1 = \max\{-\infty,1\}$$
$$-\infty \odot 1 = -\infty = -\infty + 1$$
In short, $\oplus$ can always been interpreted on $\T$ as $\max$ and $\odot$ can always be interpreted as $+$.

For $\T$ to play the role of the field of definition for a branch of algebraic geometry, it should be a field. It turns out that the tropical numbers are not a field because of the lack of additive inverses.

\begin{dfn}
A monoid with all of the additional properties of a field (except the existence of additive inverses) is a called a \textbf{semifield}.
\end{dfn}

As long as we avoid subtraction, we will be able to proceed.

\begin{thm}
With $\oplus$ and $\odot$ defined as in \ref{troperations}, $(\T, \oplus, \odot)$ is a semifield.
\end{thm}

\begin{proof}
Almost every property is known for values in $\R$, and the checks using $-\infty$ are straight-forward. Note that $0_\T = -\infty$ and $1_\T = 0$.
\end{proof}

\begin{note}
The tropical semifield is idempotent: for all $a \in \T$, $a \oplus a = a$. 
\end{note}

A common first question in tropical geometry seeks an explanation for the choice of the word ``tropical''. According to Jean-Eric Pin \cite{Pin}, the term was coined by Dominique Perrin, a French computer scientist, in honor of his Brazilian friend and colleague, Imre Simon. In the words of Maclagan and Sturmfels \cite{MS}, this choice ``simply stands for the French view of Brazil... without any deeper meaning".

According to Cohen, Gaubert, and Quadrat in \cite{CGQ}, the history of max-plus algebras can be traced to at least the early 1960s. They list scheduling theory, graph theory, dynamic programming, optimal control theory, asymptotic analysis, and discrete event theory as fields that have given rise to the use of this idempotent algebraic structure.

\begin{note}
Some of the researchers listed above actually worked with min-plus algebras on $\R \cup \{\infty\}$. The algebraic structure of these two semifields can be shown to be identical through the transformation $x \mapsto -x$.
\end{note}

Tropical polynomials behave differently from classical polynomials; two distinct tropical polynomials can produce the same tropical function.

\begin{example}
Consider the polynomial $p(x,y) = ``x^2 + xy + y^2" = x \odot x \oplus x \odot y \oplus y \odot y = \max\{2x,x+y,2y\}$. Note that $x+y > 2x$ exactly when $y > x$, in which case the third term achieves the maximum: $\max\{2x,x+y,2y\} = 2y$. The possibility $x+y > 2y$ is similar. As a result, the middle term can never be the maximum (without agreeing with another term), so it can be removed from the polynomial without changing $p(x,y)$ as a function; $``x^2 + xy + y^2" = ``x^2 + y^2"$ as functions.
\end{example}

Tropical polynomials do not produce distinct functions by evaluation; however, they can be grouped into equivalence classes producing the same function by evaluation. A version of the fundamental theorem of algebra holds for equivalence classes of tropical polynomials.

\begin{thm}\cite{MS}
Every tropical polynomial in one variable is equivalent to a tropical polynomial (as a function) that can be written as the product of linear tropical polynomials. In other words, the tropical semifield is algebraically closed.
\end{thm}

We interpret this result to mean that it is reasonable to attempt to do geometry over $\T$ that is analogous to classical geometry.

\subsection{Tropical Varieties}\label{TropVars}
\textit{In this subsection, we give the definition of a tropical variety and specify it to dimension $1$, curves.}

Tropical geometry is sometimes thought of as either a logarithmic or valuation image of classical geometry (as seen in the transformation of multiplication into addition). It is also sometimes seen as a ``dequantization" of the algebraic structure on $\R$ \cite{MikM,Litvinov}. As a result, the tropical analogue of classical objects have $\R$-dimension equal to their $\C$-dimension, and they are a special kind of polyhedral objects.

\begin{dfn}
The \textbf{zero locus} of a tropical polynomial is its non-linearity locus as a function.
\end{dfn}

\begin{example}\label{tripodex}
Consider the tropical polynomial in two variables, $p(x,y) = x \oplus y \oplus 0 = \max\{x,y,0\}$. This maximum is non-linear exactly when two of these terms agree and achieve that maximum. The zero locus of $p(x,y)$ is then the following three rays meeting at the origin.
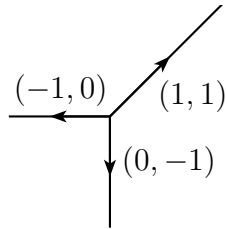
\begin{center}
\psset{xunit=0.8cm,yunit=0.8cm,algebraic=true,dotstyle=o,dotsize=3pt 0,linewidth=0.8pt,arrowsize=3pt 2,arrowinset=0.25}
\begin{pspicture*}(-1.68,-1.85)(2.68,1.84)
\psplot{0}{2.68}{(-0--1*x)/1}
\psline(0,0)(0,-1.85)
\psplot{-1.68}{0}{(-0-0*x)/-1}
\psline{->}(0,0)(1,1)
\psline{->}(0,0)(0,-1)
\psline{->}(0,0)(-1,0)
\rput[tl](0.8,0.6){$(1,1)$}
\rput[tl](-1.6,0.7){$(-1,0)$}
\rput[tl](0.2,-0.5){$(0,-1)$}
\end{pspicture*}
\end{center}
\end{example}

\begin{remark}
If there is any validity to the analogy between classical geometry, then this example can already specify several of the connections. First, the zero locus of a linear polynomial in two variables should be a genus zero curve. So this ``tripod" should be a tropical curve of genus zero. Second, notice that a $1$-dimensional tropical variety would then have $\R$-dimension $1$ as well; in general dimension $k$ tropical objects should have pieces that are dimension $k$ over $\R$. And third, notice that the three primitive integer vectors listed on the rays sum to the zero vector. So perhaps tropical varieties are going to be piece-wise linear objects with a similar condition at the intersections of the linear pieces.
\end{remark}

\begin{remark}
Mikhalkin has shown that, in general, a tropical hypersurface in $\T^n$ (the zero set of a single tropical polynomial in $n$ variables) is a polyhedral complex with a balancing condition at the $(n-2)$-dimensional faces. See Property 3.2 in \cite{MikT}.
\end{remark}

We now give Mikhalkin's more careful definition of this balancing condition in order to define tropical varieties in general. The following definition is extremely technical, but it boils down to a `zero-tension' condition that guarantees the well-definedness of degree for tropical objects.

\begin{dfn}\cite{MikM}
Let $P$ be a $k$-dimensional polyhedral complex embedded in $\T^N$; to each $k$-dimensional cell, associate a rational number called the weight. Let $F \subset P \cap \R^n$ be a $(k-1)$-dimensional cell of $P$ and $F_1,\ldots, F_l$ be the $k$-dimensional cells of $P$ adjacent to $F$ whose weights are $w_1,\ldots,w_l$. Let $L \subset \R^N$ be an $(N-k)$-dimensional affine-linear space defined by integer equations such that it intersects $F$. For a generic (real) vector $v \in \R^N$ the intersection $F_j \cap (L + v)$ is either empty or a single point. Let $\Lambda_{F_j} \subset \Z^N$ be the integer vectors parallel to $F_j$ and $\Lambda_L \subset \Z^N$ be the integer vectors parallel to $L$. Set $\lambda_j$ to be the product of $w_j$ and the index of the subgroup $\Lambda_{F_j}+\Lambda_L \subset \Z^N$. We say that $P \subset \T^N$ is \textbf{balanced} if for any choice of $F$, $L$ and a small generic $v$ the sum
$$\iota_L = \sum_{j | F_j\cap(L+v)\neq\emptyset} \lambda_j$$
is independent of $v$. We say that $P$ is \textbf{simply balanced} if in addition for every $j$ we can find $L$ and $v$ so that $F_j \cap (L + v) \neq \emptyset$, $\iota_L = 1$ and for every small $v$ there exists an affine hyperplane $H_v \subset L$ such that the intersection $P \cap (L + v)$ sits entirely on one side of $H_v + v$ in $L + v$ while the intersection $P \cap (H_v + v)$ is a point.
\end{dfn}

\begin{dfn}
Let $Y$ be a subset of $\T^N$. Then we can define a sheaf of functions on $Y$, $\O_Y$, by taking the restrictions of the (tropical) Laurent polynomials in $N$ variables to $Y$ (and its open sets).
\end{dfn}

\begin{dfn}\cite{MikM}
A topological space $X$ enhanced with a sheaf of tropical functions $\O_X$ is called a \textbf{(smooth) tropical variety} of dimension $k$ if, for every $x \in X$, there exists an open set $U \ni x$ and an open set $V$ in a simply balanced polyhedral complex $Y \subset \T^N$ such that the restrictions $\O_X|_U$ and $\O_Y|_V$ are isomorphic.
\end{dfn}

\begin{remark}\cite{MikM}\label{tropcurvedfn}
The definition of a \textbf{tropical curve}, a smooth tropical variety of dimension $1$, is particularly simple. As in \ref{tripodex}, simply balanced $1$-dimensional objects will be unions of line segments meeting at points. The sheaf of functions induces a structure that is equivalent to a complete metric away from the $1$-valent vertices and adjacent edges. So the edges without degree $1$ vertices must have finite length and are called \textit{bounded} edge; the edges containing $1$-valent vertices must have infinite length in this metric and are called \textit{unbounded} edges. In short, a tropical curve is a certain kind of metric graph.
\end{remark}

\section{Tropical Hurwitz Geometry}
\textit{In this section, we review some of what is known about the relationship between classical and tropical Hurwitz numbers as a well as motivate the kinds of geometric goals we seek for the rest of this document.}

\subsection{Classical Hurwitz Geometry}
\textit{There are many ways to compute the Hurwitz numbers classically, and it's not clear which would make the best analogy for tropical geometry. In this subsection, we frame them as the degree of a morphism between two moduli spaces.}

Hurwitz numbers lie at the intersection of many different areas of mathematics. Historically, Hurwitz related them to the representation theory of $S_d$ and thereby combinatorics. Ekedahl-Lando-Shapiro-Vainshtein connected Hurwitz numbers to the intersection theory on the moduli space of curves, Faber's Conjecture, and integrable systems (\cite{GJV}, \cite{ELSV}). More recently, new light has been shed on Hurwitz numbers by mathematical physics. In particular, they have applications in string theory, the study of Calabi-Yau manifolds, and Gromov-Witten theory. In addition, the set of (classical) ramified covers can be given a geometric structure, which we explain now for motivation both of the methods used below and the results.

Consider a collection of interesting `objects' and a notion of isomorphism of those objects. Let $M$ be the set of isomorphism classes of these objects. Sometimes it is possible to give $M$ a geometric structure. For example, the set of lines through the origin in $\R^2$ is called $\R\p^1$ and given a geometry that makes it diffeomorphic to $\s^1$. In general, projective space is given a geometric structure in a similar manner. Similarly, the tangent space to a point $x$ in a manifold is the set of parameterized curves through $x$ up to reparamaterization. This set is usually given the structure of a vector space, which has both algebraic and geometric structure. The set, $M$, of isomorphism classes, along with a geometric structure is called a \textbf{moduli space}, which I find to be one of the most interesting and exciting concepts in algebraic geometry. Sometimes the geometric structure comes naturally from the representations of the objects (or isomorphism classes). For example, the set of degree $1$ polynomials with real coefficients can be seen as a vector space. These polynomials are usually written $\{ax+b|a,b\in\R\}$; the values of $a$ and $b$ can vary, and we think of them as the coordinates on the moduli space. These varying coordinates are sometimes called parameters or \textit{moduli}.

If the geometric structure on $M$ is natural, it may be possible to use its geometry to ask global questions about the set $M$. For example, the question ``How many curves are there passing through this set of points in the plane?" is answered by Kontsevich's Formula, which can be shown using intersection theory on a moduli space. Each index condition (like passing through a fixed point) is realized as a divisor, and the computation for the formula can be shown using intersection theory. Because this proof uses inersections and degenerations, it requires that the geometric structure on the moduli space be compact, a theme that will resurface below.

In classical algebraic geometry, the most famous moduli spaces are $\overline{M}_g$, the moduli space of stable curves of genus $g$, and the related space $\overline{M}_{g,n}$, the moduli space of stable genus $g$ curves with $n$ distinct marked points. The notion of a ramified cover can also be expanded slightly to allow the construction of a moduli space of admissible covers, $\overline{H}(\overline{\sigma})$, called a Hurwitz space. The geometry of these spaces can be used to compute the Hurwitz numbers; there is a natural map $\phi: \overline{H}(\overline{\sigma}) \to \overline{M}_{g,n}$ such that the $\deg(\phi)$ encodes the Hurwitz numbers. In tropical geometry, Mikhalkin has constructed a moduli space of genus $0$ marked curves, $\M_{0,n}$, and the goal of this document is to construct a candidate for the tropical analogue of a Hurwitz space, $\H(\overline{\sigma})$, along with a natural map to $\M_{0,n}$ whose degree also encodes the Hurwitz numbers in the same way.

There is a partially understood map from classical geometry to tropical geometry called \textbf{tropicalization}, which is discussed very briefly at the beginning of subsection \ref{TropVars}. Tropicalization clearly loses some information, but it appears to retain certain critical enumerative information. Ideally, we would like to fill in the following commutative diagram in which the horizontal maps represent tropicalization and the vertical maps have degrees that encode the Hurwitz numbers by constructing $\H(\overline{\sigma})$.
$$\begin{CD}
\overline{H}(\overline{\sigma}) @>>> \H(\overline{\sigma})\\
@VVV @VVV\\
\overline{M}_{0,n} @>>> \M_{0,n}
\end{CD}$$
The relationship between $\overline{M}_{0,n}$ and $\M_{0,n}$ will be highly instructive for motivating our construction, as will understanding the structure of $\overline{H}(\overline{\sigma})$. Both of these tools require and understanding of $\overline{M}_{0,n}$, so we start there.

\begin{dfn}
An \textbf{$n$-marked rational curve} is a copy of $\p^1$ with a list of $n$ distinct points on $\p^1$, $\overline{Q} = \{Q_1,\ldots,Q_n\}$. Two $n$-marked curves $(\p^1,\overline{Q})$ and $(\p^1,\overline{Q}')$ are isomorphic if there is an isomorphism of curves $f: \p^1 \to \p^1$ such that $f(Q_i) = Q_i'$. For $n \geq 3$, let $M_{0,n}$ be the set of isomorphism classes of $n$-marked rational curves.
\end{dfn}

An $n$-marked rational curve is difficult to draw on a chalk board. As a result, many algebraic geometers draw a line segment for the $\p^1$ and a small tick mark for each marked point. Here is an image of a $4$-marked rational curve in this notation.

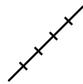
\begin{center}
\psset{xunit=1.0cm,yunit=1.0cm,algebraic=true,dotstyle=o,dotsize=3pt 0,linewidth=0.8pt,arrowsize=3pt 2,arrowinset=0.25}
\begin{pspicture*}(-.1,-.1)(1.1,1.1)
\psline(0,0)(1,1)
\psline(.15,.25)(.25,.15)
\psline(.35,.45)(.45,.35)
\psline(.55,.65)(.65,.55)
\psline(.75,.85)(.85,.75)
\end{pspicture*}
\end{center}

The set $M_{0,n}$ can be given the structure of a non-compact variety. The locations of the $Q_i$ are coordinates on this space, and there are clearly holes at the points where multiple marked points would coincide. We could compactify by allowing the points to collide; then the compactification would be $(\p^1)^n/Aut(\p^1)$. Some of these curves have an infinite number of automorphisms, which we wish to avoid here. In addition, we are going to add other information to each of the marked points, so we will not want them to be allowed to coincide. One way to do this would be to remember the direction from which the points approached each other. The projectivized tanget space to a point on $\p^1$ is again $\p^1$, so this could be accomplished by keeping the original $\p^1$, adding a second $\p^1$ intersecting the first at the point where the marked points were going to collide, and moving the colliding marked points onto the new $\p^1$. For example, if two of the marked points in $M_{0,4}$ collide, this would produce the following \textit{degeneration}.

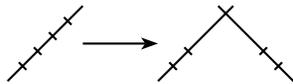
\begin{center}
\psset{xunit=1.0cm,yunit=1.0cm,algebraic=true,dotstyle=o,dotsize=3pt 0,linewidth=0.8pt,arrowsize=3pt 2,arrowinset=0.25}
\begin{pspicture*}(-.1,-.1)(4.1,1.1)
\psline(0,0)(1,1)
\psline(.15,.25)(.25,.15)
\psline(.35,.45)(.45,.35)
\psline(.55,.65)(.65,.55)
\psline(.75,.85)(.85,.75)
\psline{->}(1,0.5)(2,0.5)
\psline(2,0)(3,1)
\psline(2.15,.25)(2.25,.15)
\psline(2.35,.45)(2.45,.35)
\psline(2.8,1)(3.8,0)
\psline(3.55,.15)(3.65,.25)
\psline(3.35,.35)(3.45,.45)
\end{pspicture*}
\end{center}

This object is a \textbf{tree} of $\p^1$s, meaning that each irreducible component is a copy of $\p^1$. For the curves in $\overline{M}_{0,n}$ this tree has no cycles or self-intersections. In other words, the genus has not increased. There are now two kinds of \textbf{special} points on the irreducible components, marked points and intersections with other components.

We use these ideas to expand the collection of isomorphism classes that we are considering in order to compactify $M_{0,n}$.

\begin{dfn}
An \textbf{$n$-marked stable rational curve} is a tree of $\p^1$s, $X$ with a list of $n$ distinct, smooth points on $X$, $\overline{Q} = \{Q_1,\ldots,Q_n\}$ such that each irreducible component has at least three special points. Two $n$-marked stable curves $(X,\overline{Q})$ and $(X',\overline{Q}')$ are isomorphic if there is an isomorphism of curves $f: X \to X'$ such that $f(Q_i) = Q_i'$. For $n \geq 3$, let $\overline{M}_{0,n}$ be the set of isomorphism classes of $n$-marked stable rational curves.
\end{dfn}

The set $\overline{M}_{0,n}$ can be given the geometric structure of a compact variety.

\begin{remark}
The condition $n \geq 3$ clearly implies that an $n$-marked rational curve is an $n$-marked stable rational curve.
\end{remark}

\begin{remark}
The term \textbf{stable} is used because these curves do not have any non-trivial automorphisms. Recall that Mobius transformations (the automorphisms of $\p^1$) can take any three points on $\p^1$ to any other three points on $\p^1$. Once three points are fixed, the other points are determined.
\end{remark}

\begin{remark}
By the discussion above, $M_{0,4}$ can be given the following geometric structure. Use Mobius transformations to send the first three marked points to $\{0,1,\infty\}$. This sends the fourth point to the cross ratio of the original $4$ points. This fourth point can be located at any other point on $\p^1$, so $M_{0,4}$ can be given the geometric structure of a thrice punctured sphere. This object is sometimes called a `bowling ball'. Notice that it is homeomorphic to the `pair of pants' from $2$-dimensional topological quantum field theory.

There are three extra $4$-marked stable curves added to compactify this space, and they correspond to the three ways to add labels to the marked points in the image of a degenerate tree of $\p^1$s. As a result, $\overline{M}_{0,4}$ is isomorphic to $\p^1$.
\end{remark}

\begin{dfn}
The \textbf{boundary} of this compactification is the set of isomorphism classes of the $n$-marked stable rational curves that are not $n$-marked rational curves.
\end{dfn}

For $n > 4$, the boundary will be more than a finite collection of points. Moreover, $\overline{M}_{0,n}$ can be stratified by the number of irreducible components in the curve $X$.  The $n$-marked rational curves form a dense open set of smooth curves. Moreover, the closure of each stratum contains all of the more degenerate strata.

Now we return to ramified covers. The set of isomorphism classes of ramified covers can be given the structure of a non-compact moduli space. This set can be realized as a dense open set in a compact moduli space, the moduli space of admissible covers.

\begin{dfn}
Given a ramification profile $\overline{\sigma} = \{\sigma_1,\ldots,\sigma_n\}$ of $n$ integer partitions of $d$, an admissible cover is a morphism of curves $f: D \to X$ such that
\begin{itemize}
\item $X$ is an $n$-marked stable rational curve,
\item $f^{-1}(X_{non-sing}) = D_{non-sing}$,
\item for each irreducible component $X_i$ of $X$, $f: f^{-1}(X_i) \to X_i$ is a ramified cover,
\item $\sigma(Q_i) = \sigma_i$,
\item if $P$ lies in the intersection of two irreducible components of $D$, then the ramification index, $m_P$, is independent of the component used to compute it, and
\item at all other points $f$ is unramified.
\end{itemize}
Two admissible covers are isomorphic if there is an isomorphism of $D$ and $X$ that commutes with $f$. Given a ramification profile $\overline{\sigma}$ let $\overline{H}(\overline{\sigma})$ be the set of isomorphism classes of admissible covers, called a \textbf{Hurwitz space}.
\end{dfn}

\begin{remark}
In section 3.G of \cite{HM}, Harris and Morrison argue that $\overline{H}(\overline{\sigma})$ is a compactification of the space of ramified covers.
\end{remark}

\begin{remark}
As above, the degree $d$ is implicit in $\overline{\sigma}$, and the genus of the domain curve can be computed from $\overline{\sigma}$ using the Riemann-Hurwitz formula. If we instead dropped the requirement that all of the branch points are included in $\overline{Q}$, then this information would be nontrivial.
\end{remark}

\begin{remark}
If $X$ is smooth, this definition restricts to our definition of a ramified cover above. If $X$ is not smooth, then this definition still gives us a collection of marked points giving the ramification profile $\overline{\sigma}$. In addition, the penultimate requirement tells us that there is also a consistent notion of the ramification profile over the non-singular points of $X$. In other words, each special point in the codomain of an admissible cover has a ramification profile, not just the marked points.
\end{remark}

While the codomain curves in admissible covers have no non-trivial automorphisms, the covers themselves may. However, the groups of automorphisms for each cover will be finite. In this situation, the appropriate structure to consider is a stack. The set $\overline{H}(\overline{\sigma})$ can be given the structure of a Deligne-Mumford stack; see for example \cite{Cavalieri}. The isomorphism classes of ramified covers form a dense open set, as with $\overline{M}_{0,n}$.

There is a natural ``forgetful" morphism, $\phi_{\overline{\sigma}}: \overline{H}(\overline{\sigma}) \to \overline{M}_{0,n}$ that sends a cover $f: D \to X$ to the $n$-marked curve $X$. Notice that the boundary of the Hurwitz space maps into the boundary of the moduli space of marked points. The inverse image under $\phi_{\overline{\sigma}}$ of a marked curve $C$ is exactly the set of branched covers of $C$ with ramification profile $\overline{\sigma}$, up to certain reorderings of the marked points. As a result, the Hurwitz numbers can be recovered geometrically from the degree of this forgetful morphism. Moreover, the fact that this map even has a degree tells us that the Hurwitz number does not depend on the isomorphism class of $X$, which is not clear from the definition.

\begin{thm}\cite{ELSV,GV}
Suppose $g, m$ are integers ($g \geq 0$, $m \geq 1$) such that $2g-2+m > 0$ (ie the functor $\overline{M}_{g,n}$ is represented by a Deligne-Mumford stack). Then
$$H^g_\alpha = \frac{r!}{\# Aut(\alpha)} \prod_{i=1}^m \frac{\alpha_i^{\alpha_i}}{\alpha_i!} \int_{\overline{M}_{g,n}} \frac{1- \lambda_1 + \cdots \pm \lambda_g}{\prod (1-\alpha_i\psi_i)}$$
where $\lambda_i = c_i(\mathbb{E})$ ($\mathbb{E}$ is the Hodge bundle).
\end{thm}

The key tool in the proof of the above theorem in \cite{ELSV} relies on the famous Lyashko-Looijenga mapping, $\LL$, that associates to a holomorphic function the (unordered) set of its ramification values, much like the branch morphism. The factor $Aut(\overline{\sigma})$ below simply accounts for issues arising from the situation in which multiple points have the same ramification profile. The following theorem frames the classical results in the form most analogous to the tropical results below.

\begin{thm}\cite{ELSV}
There is a morphism $\LL: H(\overline{\sigma}) \to \C^{\mu-1}$ so that $h(\overline{\sigma}) = \frac{\deg(\LL)}{|Aut(\ol{\sigma})|}$.
\end{thm}

\subsection{Tropicalization}

Recall from example \ref{tripodex} that our tropical analogue of $\p^1$ was a graph with a single vertex and three emanating rays. This inspires \textit{both} a dual picture to our trees of $\p^1$s above and a possible way to think of tropicalization.

\begin{dfn}
Consider an $n$-marked stable rational curve $X$. Create a graph $G_X$ in the following manner. For each irreducible component and marked point of $X$, create a vertex for $G_X$. If a marked point is on an irreducible component, add an edge to $G_X$ between their vertices. If two irreducible components of $X$ intersect, add an edge between their vertices. We will call $G_X$ the \textbf{dual graph of $X$}.
\end{dfn}

\begin{remark}
Notice that the stability condition implies that the vertices corresponding to irreducible components have degree at least $3$. This condition will still be called stability in the following chapters. Also notice that the definition of a tree of $\p^1$s tells us that the image is connected and has no circuits, namely it is a tree. 
\end{remark}

We now compare the image of the dualization map to Mikhalkin's moduli space of marked tropical rational curves, $\M_{0,n}$. Recall that a tropical curve is a metric graph with a complete metric on the complement of the 1-valent vertices (\ref{tropcurvedfn}).

\begin{dfn}
The genus of a tropical curve $X$ is $g= \dim(H^1(X))$.
\end{dfn}

\begin{remark}
So, the connected, genus zero tropical curves will have no circuits, meaning that they are trees.
\end{remark}

\begin{remark}\cite{MikM}
Adding an unbounded edge to a tropical curve is the tropical analogue of deleting a point from the dual classical curve. This is an example of a more general technique, called \textbf{tropical modification}. See \cite{MikM} for more detail in higher dimensions. As a result, we can think of adding unbounded edges on a tropical curve as analogous to marking points on a classical curve. As a result, adding unbounded edges to a tropical curve gives us a way to mark points on that curve.
\end{remark}

\begin{dfn}
Let $\M_{g,n}$ be the set of tropical curves with genus $g$ and $n$ distinct marked points (meaning $n$ distinct unbounded edges).
\end{dfn}

\begin{remark}
The only automorphism of a tree that fixes all of its leaves is the identity. To avoid automorphisms of a curve with itself, we and Mikhalkin restrict ourselves to the case $g=0$. In addition, the moduli spaces of higher genus curves are much less well understood.
\end{remark}

Mikhalkin then gives a description of $\M_{0,n}$ as a polyhedral complex.

\begin{dfn}\label{combtype}
The \textbf{combinatorial type} of a tropical curve with $n$ marked points is its equivalence class up to homeomorphisms respecting the markings.
\end{dfn}

The combinatorial types partition the set $\M_{0,n}$ into disjoint subsets. The edge-length functions for the finite edges give the subset of $\M_{0,n}$ with a given combinatorial type the structure of an integral polyhedral cone $\R^M_{>0}$ because each of the lengths must be positive, where $M$ is the number of finite edges. If $G$ is a genus $0$ curve with only degree $1$ and $3$ vertices, then $G$ has $n-3$ bounded edges (and fewer otherwise).

Furthermore, any face of the closure of this cone, $\R^M_{\geq0}$, gives a cone corresponding to another combinatorial type, the type in which some of the edges of the curve have been contracted. This gives an adjacency structure on $\M_{0,n}$, making it a polyhedral complex.

\begin{thm}
The set $\M_{0,n}$ for $n \geq 3$ admits the structure of an $(n-3)$-dimensional tropical variety such that the edge-length functions are regular within each combinatorial type. Furthermore, $\M_{0,n}$ can be tropically embedded in $\R^N$ for some $N$, meaning that $\M_{0,n}$ can be presented as a simply balanced complex.
\end{thm}

Although tropicalization is not understood fully, there are some relationships known for the map $\overline{M}_{0,n} \to \M_{0,n}$.

\begin{example}
Classically, $\overline{M}_{0,4}$ is isomorphic to $\p^1$, which we saw above. Mikhalkin's construction gives $\M_{0,4}$ the structure of three rays emanating from a point. This is exactly our candidate for the tropical analogue of $\p^1$! Moreover, each of the rays in $\M_{0,4}$ is associated to one of the three points in the boundary of $\overline{M}_{0,4}$. The position on the ray corresponds to the length of the bounded edge in the dual graph to these degenerate trees of $\p^1$s.
\end{example}

For now, this discussion allows us to think of tropicalization as dualization. Notice that dualization is a ``degeneration reversing" map between the classical and tropical moduli spaces of marked curves.

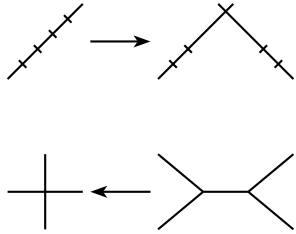
\begin{center}
\psset{xunit=1.0cm,yunit=1.0cm,algebraic=true,dotstyle=o,dotsize=3pt 0,linewidth=0.8pt,arrowsize=3pt 2,arrowinset=0.25}
\begin{pspicture*}(-0.1,-2.1)(4.1,1.1)
\psline(0,0)(1,1)
\psline(.15,.25)(.25,.15)
\psline(.35,.45)(.45,.35)
\psline(.55,.65)(.65,.55)
\psline(.75,.85)(.85,.75)
\psline{->}(1.1,0.5)(1.9,0.5)
\psline(2,0)(3,1)
\psline(2.15,.25)(2.25,.15)
\psline(2.35,.45)(2.45,.35)
\psline(2.8,1)(3.8,0)
\psline(3.55,.15)(3.65,.25)
\psline(3.35,.35)(3.45,.45)
\psline{<-}(1.1,-1.5)(1.9,-1.5)
\psline(0,-1.5)(1,-1.5)
\psline(0.5,-2)(0.5,-1)
\psline(2.6,-1.5)(3.2,-1.5)
\psline(2.6,-1.5)(2,-1)
\psline(2.6,-1.5)(2,-2)
\psline(3.2,-1.5)(3.8,-1)
\psline(3.2,-1.5)(3.8,-2)
\end{pspicture*}
\end{center}

\begin{prop}
The map $X \to G_X$, called dualization above, gives a bijection between the boundary strata of $\overline{M}_{0,n}$ and the combinatorial types of $n$-marked, genus $0$ tropical curves. Moreover, if $Y$ is a stratum of the boundary of $\overline{M}_{0,n}$ and $Y'$ is a stratum in the closure of $Y$, then the combinatorial type of $G_Y$ is a degeneration of the combinatorial type of $G_{Y'}$.
\end{prop}

\begin{proof}
The most general curves classically are smooth and hence dualize to a graph with a single vertex of degree $n$ with $n$ unbounded edges. This tropical curve is a degeneration of all combinatorial types. The most degenerate classical curves have exactly three special points, be they marked points or intersections. The dual of such a degenerate classical curve will be a tree in which every vertex is degree $1$ or $3$, a condition that we will later call \textit{trivalence}.

Classical degeneration involves the collision of at least two special points from a component with at least $4$; these colliding points end up on a new irreducible component. The dualization of this process picks a vertex of degree at least $4$, splits its edges between two vertices and adds an edges between these vertices. This description is exactly the reverse of a degeneration of a stable tropical curve.
\end{proof}

\begin{remark}
All of the positive dimensional cones in Mikhalkin's construction already correspond to the boundary of the classical moduli space. As a result, the polyhedral complex constructed below will be open, not compact.
\end{remark}

Finally, we are ready to define tropical ramified covers. Taking the analogy from classical geometry, we will want to talk about tropical covers of stable graphs with ramification profiles at the degree $1$ vertices (the marked points). In addition, classical admissible covers allowed for a ramification profile at the intersection point. This point has become the edge between the high-degree vertices, so we will also want to assign ramification profiles to each of them.

\begin{dfn}
Consider $\overline{\sigma} = \{\sigma_1,\ldots,\sigma_n\}$, where each $\sigma_i$ is an integer partition of $d$. Pick a tropical curve $G$ with $n$ marked points; then a \textbf{tropical ramified cover} with ramification profile $\overline{\sigma}$ will be a copy of $G$ with the following integer partitions associated to edges of $G$. Associate $\sigma_i$ to the $i^{\mbox{th}}$ marked point; also assign an integer partition of $d$ to each of the finite edges of $G$.
\end{dfn}

\begin{remark}
This definition of a tropical ramified cover lists the codomain curve and the ramification profile, which seems substantively different from the classical definition of a ramified cover because we do \textit{not} actually include the tropical covering map. This leaves open the question of whether the information included above is sufficient for specifying a more analogous definition of a ``tropical admissible cover" and whether that more geometric version of a ``tropical admissible cover" can be realized as the tropicalization of the classical cover with the same data. The authors of \cite{CJM} give such a definition in the case of only two specified marked points, and their computations do use exactly the information that I specify here.

In the maximally degenerate situation of a trivalent graph for the tropical curve, each vertex has exactly three edges, meaning three ramification points. There is a unique such cover classically, and the required gluing is also specified for us, so this claim seems reasonable.

However, it \textit{is} clear how to get one of our tropical ramified covers from an admissible cover. Simply take the marked stable curve that is its codomain and dualize it to a graph. To each marked point, assign the integer partition given by the ramification profile over that point. To each finite edge, associate the ramification profile coming from the ramification over the associated non-singluar point.
\end{remark}

\textit{In summary:} There is already a tropical moduli space of genus $0$ curves with $n$ marked points, $\M_{0,n}$, which was created by Mikhalkin (\cite{MikM}). The goal of this document is, for each ramification profile $\overline{\sigma}$, to give the set of tropical ramified covers the structure of a moduli space, $\H(\overline{\sigma})$, such that
\begin{itemize}
\item $\H(\overline{\sigma})$ is a connected, polyhedral complex, and
\item $\H(\overline{\sigma})$ has a natural morphism of polyhedral complexes $\phi: \H(\overline{\sigma}) \to \M_{0,n}$ such that $\deg(\phi)$ encodes the Hurwitz number $h(\overline{\sigma})$.
\end{itemize}

\chapter{Preliminaries}
\textit{We establish the definitions and lemmas needed for the original mathematics in the following chapter.}

\section{Graphs}
\textit{Having reframed the discussion of tropical curves in the previous section in terms of metric graphs, we must solidify the foundations of our perspective in the language of graph theory.}

\begin{notation}
A graph is an ordered pair $G= (V,E)$, with $V$ a finite set (called the \textbf{vertices}) and $E$ a set of (unordered) pairs of elements from $V$ (called the \textbf{edges}). Note that we are disallowing ``multiple edges" in graphs. The vertices appearing in an edge are called its \textbf{endpoints}. A \textbf{tree} is a connected graph without circuits. The \textbf{degree} of a vertex $v \in V$ is the number of times $v$ appears as an endpoint of elements in $E$. Vertices of degree $1$ are called \textbf{leaves}.
\end{notation}

\begin{dfn}
Edges containing a leaf will be called \textbf{unbounded edges}; the other edges will be called \textbf{bounded edges}. Vertices that are not leaves will be called \textbf{internal vertices}.
\end{dfn}

\begin{thm}\label{leavesexist}
Any tree with at least one edge has at least two leaves.
\end{thm}

\begin{proof}
Notice that removing any edge from $T$ disconnects it into two trees. This standard result can be shown by removing any edge and using strong induction on the number of edges in the tree, starting with the unique tree with $1$ edge (and two leaves).
\end{proof}

\begin{notation}\label{partitionpi}
For a tree $T$, let $L(T)$ be the set of leaves of $T$. For any edge $e$ of $T$, let $\pi(e) = \{S(e),S'(e)\}$ be the (set) partition of $L(T)$ obtained by deleting $e$ from $T$ and partitioning the leaves by the connected component in which they lie.

\begin{remark}
Notice also that \ref{leavesexist} implies that neither $S(e)$ nor $S'(e)$ can be empty. Either $e$ is an unbounded edge, in which case it clearly separates one leaf from the others, or $e$ is bounded, and the components have other edges and hence at least $2$ leaves each.
\end{remark}

Let $T = (V,E)$ and $T' = (V',E')$ be trees, and let $f: V \to V'$ be a bijection. Then the function $f$ extends to subsets of $L(T)$, writing $f(\pi(e)) = \{f(S(e)),f(S'(e))\}$.
\end{notation}

\begin{dfn}
A graph is said to be \textbf{stable} if it has no vertices of degree $2$.
\end{dfn}

\begin{remark}
This definition of stable is the tropicalization (dualization) of the classical definition of stability. In addition, degree $2$ vertices allow for a infinite number of metric structures, akin to the infinite number of automorphisms of $\p^1$ with only $2$ marked points.
\end{remark}

\begin{lemma}\label{distinctpartitions}
Let $T = (V,E)$ be a stable tree with $e, e' \in E$. If $e \neq e'$, then $\pi(e) \neq \pi(e')$.
\end{lemma}

\begin{proof}
If $|E| = 1$, then the lemma is vacuously true. Otherwise, pick two distinct edges $e$ and $e'$ in $E$. Find a path in $T$ containing $e$ and $e'$ as follows. If you remove $e$ from $T$, one of its endpoints is not in the component of what remains that contains $e'$; call this endpoint $v$. Similarly, let $w$ be the endpoint of $e'$ not in the component with $e$. Then the (unique) path from $v$ to $w$, $P_{(v,w)}$, contains both $e$ and $e'$. Now I will show that the partitions associated to edges in a path are all distinct, showing that $\pi(e) \neq \pi(e')$.

Consider $\pi(e) = \{S(e),S'(e)\}$, labeled so that $S(e)$ comes from the component containing $e'$. Let $v_1$ be the other endpoint of $e = \{v,v_1\}$. Because $T$ is stable, $\deg(v_1) \geq 3$, so there is at least one edge coming from $v$ that is not included in $P_{(v,w)}$, $\hat{e}_1$. Let $S(\hat{e}_1)$ be the part coming from the component not containing $P_{(v,w)}$. Then $S(\hat{e}_1)$ is a subset of $S(e)$. Moreover, $$\pi(e_1) = \{S(e) \setminus \left(\cup S(\hat{e}_1)\right),S'(e) \cup \left(\cup S(\hat{e}_1)\right)\}.$$
The figure below shows this in a simple case.
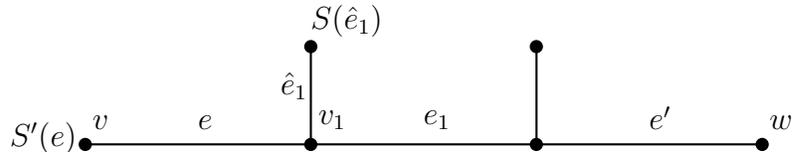
\begin{center}
\psset{xunit=1.0cm,yunit=1.0cm,algebraic=true,dotstyle=*,dotsize=5pt 0,linewidth=0.8pt,arrowsize=3pt 2,arrowinset=0.25}
\begin{pspicture*}(-5.5,-.1)(5,2)
\psline(-4.5,0)(-1.5,0)
\psline(-1.5,0)(1.5,0)
\psline(1.5,0)(4.5,0)
\psline(-1.5,0)(-1.5,1.3)
\psline(1.5,0)(1.5,1.3)
\psdots(-4.5,0)
\rput[bl](-4.4,0.2){$v$}
\psdots(-1.5,0)
\rput[bl](-1.4,0.2){$v_1$}
\psdots(1.5,0)
\psdots(4.5,0)
\rput[bl](4.6,0.2){$w$}
\rput[bl](-3,0.2){$e$}
\rput[bl](0,0.2){$e_1$}
\rput[bl](3,0.2){$e'$}
\psdots(-1.5,1.3)
\rput[bl](-1.9,0.6){$\hat{e}_1$}
\psdots(1.5,1.3)
\rput[tl](-5.5,.3){$S'(e)$}
\rput[tl](-1.5,1.85){$S(\hat{e}_1)$}
\end{pspicture*}
\end{center}
In short, at each step along the path $P_{(v,w)}$, some leaves move from $S$ to $S'$, and they move only in that direction. The only remaining concern is that we will somehow end up with a new partition with the roles of $S$ and $S'$ switched. However, the edges of $S'(e)$ are fixed, so there is no chance that $S$ and $S'$ switch roles.
\end{proof}

\begin{dfn}\label{graphisodfn}
Let $G_1 = (V_1,E_1)$ and $G_2 = (V_2,E_2)$ be graphs. A \textbf{morphism} $F: G_1 \to G_2$ is a function $F_V: V_1 \to V_2$ such that
$$F_E(\{v,w\}) := \{F_V(v),F_V(w)\}$$ is a function $F_E: E_1 \to E_2$. A morphism $F$ is an \textbf{isomorphism} if $F_V$ and $F_E$ are bijections.
\end{dfn}

\begin{dfn}\label{topotype}
The \textbf{topological type} of a stable graph is its isomorphism class in the sense of \ref{graphisodfn}.
\end{dfn}

\begin{prop}\label{graphisomorph}
Let $T_1 = (V_1,E_1)$ and $T_2 = (V_2,E_2)$ be stable trees, and let $f: L(T_1) \to L(T_2)$ be a bijection. Suppose that
\begin{equation}
\{\pi(e_2) | e_2 \in T_2\} = \{f(\pi(e_1)) | e_1 \in T_1\}.
\end{equation}
Then there exists an isomorphism of graphs, $F: T_1 \to T_2$, such that the restriction of $F_V$ to $L(T_1)$ equals $f$. In other words, $f$ can be extended to an isomorphism of graphs.
\end{prop}

\begin{proof}
First notice that \ref{distinctpartitions} tells us that the partitions coming from edges within a single graph are distinct. Under the hypotheses of equation $(2.1)$, there is a pairing of the partitions from the two graphs. This implies that $|E_1| = |E_2|$.

We will prove this lemma by induction on $n = |E_1|$, the number of edges of $T_1$. Suppose $n = 1$; then $T_1$ and $T_2$ are both the unique tree with one edge. This tree has exactly two vertices, both of which are leaves. So $F_V := f $ is already a bijection on the full set of vertices. By inspection, the bijection $F_V$ does induce the map $F_E$ that sends the one edge of $T_1$ to the one edge of $T_2$ and hence $F$ is an isomorphism of the trees.

Let $n \in \N$ with $n \geq 2$ and suppose, for any $j \in \N$ with $j < n$ the statement of our theorem is true for trees with $j$ edges. Let $T_1$ and $T_2$ be two trees satisfying equation $(2.1)$ such that $T_1$ has $n$ edges. Recall that the unbounded edges partition the leaves of a tree into a singleton and a remaining leaves. As a result, if $T_1$ does not have any bounded edges, then neither does $T_2$. In this case, they are both the star of $n$ edges, which are clearly isomorphic in a way that extends the isomorphism on the leaves.

Otherwise, choose a bounded edge $e_1 = \{v_1,v_2\}$ of $T_1$. Add two new vertices, $\hat{v}_1$ and $\hat{v}_2$ to $V_1$; remove $e_1$ from $E_1$ and replace it with two new edges, $\{v_1,\hat{v}_1\}$ and $\{v_2,\hat{v}_2\}$. This is essentially breaking the edge $e_1$ into two unbounded edges. The resulting graph is the union of two trees. Label these trees $T_{1,1}$ and $T_{1,2}$, with the leaves $S(e_1) \cup \{\hat{v}_1\}$ in $T_{1,1}$ and the leaves $S'(e_1) \cup \{\hat{v}_2\}$ in $T_{1,2}$.

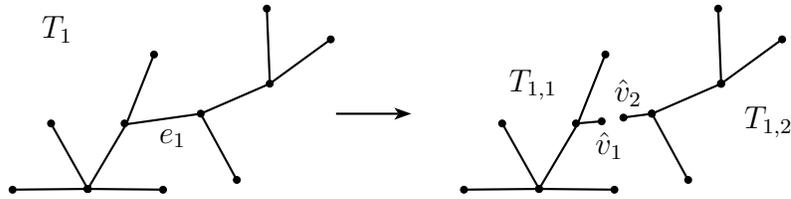
\begin{center}
\psset{xunit=1.0cm,yunit=1.0cm,algebraic=true,dotstyle=*,dotsize=3pt 0,linewidth=0.8pt,arrowsize=3pt 2,arrowinset=0.25}
\begin{pspicture*}(-1.1,-0.1)(9.4,2.5)
\psline(0,0)(-0.49,0.87)
\psline(0,0)(1,-0.01)
\psline(0,0)(0.51,0.86)
\psline(0.51,0.86)(0.88,1.79)
\psline(0.51,0.86)(1.5,1)
\psline(1.5,1)(2.42,1.4)
\psline(1.5,1)(1.98,0.12)
\psline(2.42,1.4)(2.38,2.4)
\psline(2.42,1.4)(3.23,1.99)
\psline(0,0)(-1,-0.01)
\psline(6,0)(5,-0.01)
\psline(6,0)(5.51,0.87)
\psline(6,0)(7,-0.01)
\psline(6,0)(6.51,0.86)
\psline(6.51,0.86)(6.88,1.79)
\psline(7.5,1)(8.42,1.4)
\psline(7.5,1)(7.98,0.12)
\psline(8.42,1.4)(8.38,2.4)
\psline(8.42,1.4)(9.23,1.99)
\psline{->}(3.3,1)(4.3,1)
\psline(6.49,0.87)(6.83,0.9)
\psline(7.12,0.95)(7.5,1)
\rput[tl](-0.61,2.3){$T_1$}
\rput[tl](5.6,1.59){$T_{1,1}$}
\rput[tl](8.73,1.09){$T_{1,2}$}
\psdots(0,0)
\psdots(1.5,1)
\psdots(-0.49,0.87)
\psdots(1,-0.01)
\psdots(0.88,1.79)
\psdots(2.42,1.4)
\rput[bl](.95,0.55){$e_1$}
\rput[bl](6.75,0.4){$\hat{v}_1$}
\rput[bl](7,1.1){$\hat{v}_2$}
\psdots(1.98,0.12)
\psdots(2.38,2.4)
\psdots(3.23,1.99)
\psdots(0,0)
\psdots(0.49,0.87)
\psdots(0,0)
\psdots(-1,-0.01)
\psdots(0,0)
\psdots(6.49,0.87)
\psdots(5,-0.01)
\psdots(5.51,0.87)
\psdots(6,0)
\psdots(7,-0.01)
\psdots(6,0)
\psdots(6.88,1.79)
\psdots(7.5,1)
\psdots(8.42,1.4)
\psdots(7.5,1)
\psdots(7.98,0.12)
\psdots(8.42,1.4)
\psdots(8.38,2.4)
\psdots(8.42,1.4)
\psdots(9.23,1.99)
\psdots(6.83,0.9)
\psdots(7.12,0.95)
\end{pspicture*}
\end{center}

Here is another way to think of this construction: the tree $T_{1,j}$ was formed by crushing all of $T_1$ separated from $v_j$ by $v_{(1-j)}$ to the point $v_{(1-j)}$ and calling that crushed point $\hat{v}_{(1-j)}$.

By hypothesis, there is an edge $e_2$ such that $\pi(e_2) = f(\pi(e_1))$. Notice that $e_2$ must also be bounded. As with $T_1$, break $e_2$ into two unbounded edges by adding $\hat{w}_1$ and $\hat{w}_2$ to $V_2$, labelled so that $\hat{w}_1$ lies in the component containing the leaves $f(S(e_1))$. Then delete $e_2$ from $E_2$ and add $\{w_1,\hat{w}_1\}$ and $\{w_2,\hat{w}_2\}$ to $E_2$. Label these trees $T_{2,1}$ and $T_{2,2}$, with the leaves $f(S(e_1)) \cup \{\hat{w}_1\}$ in $T_{2,1}$ and the leaves $f(S'(e_1)) \cup \{\hat{w}_2\}$ in $T_{2,2}$.

Notice that $f$ restricts to bijections $f_i: L(T_{1,i}) \setminus \{\hat{v}_i\} \to L(T_{2,i}) \setminus \{\hat{w_i}\}$; extend these functions to bijections $\hat{f}_i: L(T_{1,i}) \to L(T_{2,i})$ by setting $\hat{f}_i(v_i) = w_i$. If equation $(2.1)$ for $T_1$ and $T_2$ descends to equation $(2.1)$ for each of the pairs $T_{1,j}$ and $T_{2,j}$, then we will be able to use the inductive hypothesis.

Notice that, for any edge $e \neq e_1$ in $V_1$, $\pi(e)$ is related to $\pi(e_1) = \{S(e_1),S'(e_1)\}$. If $e$ is in $T_{1,1}$, then $e$ separates some $S(e_1)$ and moves it into $S'(e_1)$. A similar discussion holds for any of the other three new trees. This implies that the original pairing given by equation $(2.1)$ splits into two pairings for the new trees.

These original partitions are not maintained in the new trees; there are different leaves. However, the relation is straight-forward. Consider an edge $e$ from $T_1$, but think of it in $T_{1,j}$. I claim that we can compute $\pi(e)_{T_{1,j}}$ from $\pi(e)$ as follows. Notice that all of $L(T_{1,(1-j)}) \setminus \{\hat{v}_{(1-j)}\}$ lives in one part of the partition $\pi(e)$. To compute $\pi(e)_{T_{1,j}}$, simply find this collection of leaves and replace it with $\hat{v}_j$. A similar discussion holds for the edges of $T_{2,j}$. This shows that equation $(2.1)$ does descend to the appropriate equations for the pairs of new trees.

In short, removing these edges gives two pairs of sub-trees satisfying all of the hypotheses of the theorem. Moreover, each subtree $T_{1,i}$ has fewer edges than $T_1$, so by strong induction, each of $\hat{f}_i$ extends to an isomorphism $F_i: T_{1,i} \to T_{2,i}$. Taken together, these isomorphisms are a bijection $F_V: V_1 \cup \{\hat{v}_1,\hat{v}_2\} \to V_2 \cup \{\hat{w}_1,\hat{w}_2\}$ that induces a bijection $\hat{F}_E: (E_1 \cup \{\{v_1,\hat{v}_1\}, \{v_2,\hat{v}_2\}\}) \setminus \{e_1\} \to (E_2 \cup \{\{w_1,\hat{w}_1\},\{w_2,\hat{w}_2\}\}) \setminus \{e_2\}$. Notice that $F_V(\hat{v}_i) = \hat{w}_i$, so $F_V(v_i) = w_i$ so that the edge $\{v_i,\hat{v}_i\}$ is mapped correctly. Hence $F_E(e_1) = F_E(\{v_1,v_2\}) = \{F_V(v_1),F_V(e_2)\} = \{w_1,w_2\} = e_2$. Thus $F: T_1 \to T_2$ is an isomorphism.
\end{proof}

\begin{dfn}\label{degeneration}
Consider a graph $G = (V,E)$ with a bounded edge $e = \{v_1,v_2\}$. Build a new new graph $G' = (V',E')$ as follows.
\begin{itemize}
\item The set $V'$ is built from $V$ by removing $v_2$.
\item The set $E'$ is built from $E$ by removing $e$ and then replacing all instances of $v_2$ in edges with $v_1$.
\end{itemize}
Essentially, the graph $G'$ is built from $G$ by shrinking the edge $e$ to a point. As a result, we say that $G'$ is the \textbf{degeneration} of $G$ by $e$.
\end{dfn}

\begin{remark}
We chose the term degeneration because of the parallel to classical geometry, but this graph-theoretic construction is also called \textit{contraction}. Degeneration of a graph by the edge $e$ is defined by removing one of the endpoints of $e$ and \textit{renaming} certain endpoints. Instead, this definition could have been framed in terms of \textit{identifying} the two endpoints of $e$. As a result we can talk about degenerating a graph by multiple edges simultaneously without worrying about the order of operations.
\end{remark}

\begin{dfn}
A graph is said to be \textbf{trivalent} if all of its vertices have degree $1$ or $3$.
\end{dfn}

\begin{remark}
The trivalent graphs are the duals of the maximally degenerate classical curves but are the most general tropical curves.
\end{remark}

\begin{lemma}\label{n-3}
Let $G$ be a trivalent tree with $|L(G)| = n$. If $n \geq 3$, then $G$ has $n-3$ bounded edges.
\end{lemma}

\begin{proof}
The simplest trivalent tree has exactly three (unbounded) edges meeting at a point. The claim is true for this tree. Any trivalent tree can built from this tree by inserting a point in the middle of an existing edge and adding a new unbounded edge and leaf at the new internal vertex. This process adds an bounded edge and leaf at each step, so the relationship remains true for all trivalent trees.

Alternately, noticing that any tree is planar allows us to show this result using the Euler Characteristic Theorem.
\end{proof}

\begin{lemma}\label{twoleaftripod}
In any trivalent tree $G$ with $n \geq 3$ leaves, there is an internal vertex in $G$ where two distinct unbounded edges meet.
\end{lemma}

\begin{proof}
Consider the subgraph, $B$, formed from the bounded edges and internal vertices. The graph $B$ is still connected and hence a tree. If $B$ is a single vertex, then that vertex is the intersection of three unbounded edges in $G$. Otherwise, $B$ is a tree with edges and hence at least two leaves. The remaining two degrees for these internal vertices in $G$ must come from unbounded edges.
\end{proof}

\begin{lemma}\label{stabledegen}
Let $G$ be a graph. If $G$ is trivalent or stable, then any degeneration of $G$ is stable.
\end{lemma}

\begin{proof}
We only allow degenerations by bounded edges, so the degree of a leaf will never change. When degenerating by the bounded edge, $e = \{v,w\}$, the degree of the identified vertex, $v=w$, in the degeneration will be $\deg(v) + \deg(w) - 2 \geq 3 + 3 - 2 = 4$.
\end{proof}

\section{Polyhedral Complexes}
\textit{We are not going to be able to embed our polyhedral complexes in a single affine space. As a result, we will not be able to show that the combinatorial object built is a tropical variety. We now state a (generalized) definition of a polyhedral complex that does not need to be embedded in a single affine space.}

\begin{dfn}
Let $C$ be a topological space and $P$ a closed subset of $C$.  A \textbf{chart} for $P$ is a homeomorphism of $P$ with a closed, possibly unbounded lattice polyhedron $X$ in $\R^j$.
\end{dfn}

\begin{notation}
Let $X \subset \R^j$ be a lattice polyhedron. Let $V_X$ be the smallest affine space containing $X$, that is, $V_X = \{x + span\{x - x'\} \mid x,x' \in X\}$.
\end{notation}

\begin{dfn}
Let $X \subset \R^j$ and $Y \subset \R^k$ be lattice polyhedra. Two charts for $P$, $c_X: X \to P$ and $c_Y: Y \to P$, are \textbf{equivalent} if there is an isomorphism of affine spaces $f: V_X \to V_Y$ such that $f(X) = Y$ and $c_X = c_Y \circ f$. We say $c_X$ and $c_Y$ are \textbf{lattice-equivalent} if $f$ restricts to an (affine) isomorphism of $V_X \cap \Z^j$ with $V_Y \cap \Z^k$. Note that such an $f$ is unique if it exists because $X$ spans $V_X$.
\end{dfn}

\begin{dfn}\label{polycomplex}
A \textbf{polyhedral complex}, $C$, is a topological space, together with a collection of closed cells $\overline{P} = \{P_i|i \in I\}$ and for each $i$ a lattice-equivalence class of charts $\overline{c} = \{c_i: X_i \to P_i|i \in I\}$ for some closed lattice polyhedra $X_i$ such that:
\begin{itemize}
\item $C$ is the union of the $P_i$,
\item for each $i,j \in I$, the intersection $P_i \cap P_j$ is equal to $P_k$ for some $k \in I$,
\item if $Y$ is a face of $X_i$, then there exists $j$ such that $c_i(Y) = P_j$, and
\item if two charts have the same image, $P_i = P = P_j$, then $c_i: X_i \to P$ and $c_j: X_j \to P$, are lattice-equivalent.
\end{itemize}
The elements of $\overline{P}$ will be called the polyhedral cells of $P$.
\end{dfn}

\begin{dfn}
If $c: X \to P$ is a chart and $Y$ is a face of $X$, the we say that the polyhedral cell $c(Y)$ is a \textbf{face} of $P$.
\end{dfn}

\begin{dfn}
A \textbf{morphism} of polyhedral complexes is a continuous function $\phi: C \to C'$ such that
\begin{itemize}
\item any polyhedral cell $P_i$ of $P$ maps entirely into a polyhedral cell $P_j'$ of $P'$ with
\item $c_j'^{-1} \circ \phi \circ c_i: X_i \to X_j'$ affine and integral where defined.
\end{itemize}
\end{dfn}

\begin{dfn}
A polyhedral cell $P$ with chart $c: X \to P$ is said to be \textbf{dimension $k$} if the lattice polyhedron $X$ is dimension $k$. A polyhedral complex is said to have \textbf{dimension $k$} if $k$ is the biggest dimension of any of its polyhedral cells and any polyhedral cell is a face of a polyhedral cell of dimension $k$.
\end{dfn}

\begin{dfn}
Let $\phi: C \to C'$ be a morphism of polyhedral complexes of dimension $k$. Pick a point $p$ in the interior of a $k$-dimensional polyhedral cell, $P_i$. Then there exists a polyhedral cell $P_j'$ of $P'$ such that $c_j'^{-1} \circ \phi \circ c_i: X_i \to X_j'$ is affine and integral where defined. There is a lattice $\Lambda_i$ in $X_i$ and a lattice $\Lambda_j'$ in $X_j'$. Let $\ind(\phi)_p$ be $0$ if the dimension of $P_j'$ is not $k$ and otherwise let $\ind(\phi)_p = [c_j'^{-1} \circ \phi \circ c_i(\Lambda_i):\Lambda_j']$.
\end{dfn}

\begin{dfn}
A polyhedral complex of dimension $k$ will be said to be \textbf{weighted} if each polyhedral cell of dimension $k$ is assigned a rational number. The weight of a polyhedral cell $P_i$ will usually be denoted by $w(P_i)$.
\end{dfn}

\begin{remark}
Although the weights are associated to top-dimensional polyhedral cells, we can think of the weight as a function that is only defined on the interiors of the top-dimensional polyhedral cells.
\end{remark}

\begin{dfn}
Let $C$ and $C'$ be polyhedral complexes and $\phi: C \to C'$ a morphism of polyhedral complexes of the same dimension. Suppose $C$ is weighted; the weight of the polyhedral cell $P_i$ will be written $w(P_i)$. Pick $q \in C'$ in the interior of a top dimensional polyhedral cell. Define
$$\deg(\phi)_q = \sum_{\begin{array}{c}p \in P\\\phi(p) = q\end{array}} w(p) \cdot \ind(\phi)_p.$$
If the sum is independent of the choice of $q$, then we say that the \textbf{degree} of $\phi$ is this constant value, denoted $\deg(\phi)$.
\end{dfn}

\section{Orthogonal Basis Lemma}
\textit{The previous section shows us that computing the degree of a morphism will require us to simplify a sum. The following lemma will allow us to do that simplification later in the proof of the main theorem. This lemma is a special case of a general lemma about bilinear pairings with an orthogonal basis.}

\begin{prop}\label{states}
Let $\omega_1, \omega_2 \in Z(\R[S_d])$. Then
$$\tr(\omega_1\omega_2) = \sum_{\nu} \frac{1}{|K_\nu|}\tr(\omega_1 K_\nu)\tr(K_\nu \omega_2),$$
where the sum is taken over integer partitions $\nu$ of $d$.
\end{prop}

\begin{proof}
First, write $\omega_i = \sum_\alpha \omega_{i,\alpha} K_\alpha$. Then
$\sum_\nu \frac{1}{|K_\nu|}\tr(\omega_1K_\nu)\tr(K_\nu \omega_2)$
$$\begin{array}{lll}
= & \sum_\nu \frac{1}{|K_\nu|}(\sum_\alpha \omega_{1,\alpha} \tr(K_\alpha K_\nu))(\sum_\beta \omega_{2,\beta} \tr(K_\nu K_\beta)) & \mbox{linearity}\\
= & \sum_{\nu} \frac{1}{|K_\nu|}(\omega_{1,\nu} \tr(K_\nu K_\nu))(\omega_{2,\nu} \tr(K_\nu K_\nu)) & \mbox{orthogonality}\\
= & \sum_{\nu} \frac{1}{|K_\nu|}(\omega_{1,\nu} \omega_{2,\nu}) (|K_\nu|)^2 & \mbox{\ref{conjsize}}\\
= & \sum_{\nu} (\omega_{1,\nu} \omega_{2,\nu}) |K_\nu| & \mbox{}\\
= & \sum_{\nu,\epsilon} (\omega_{1,\nu} \omega_{2,\epsilon}) \tr(K_\nu K_\epsilon) & \mbox{orthogonality}\\
= & \tr( \sum_{\nu,\epsilon} (\omega_{1,\nu} \omega_{2,\epsilon}) K_\nu K_\epsilon) & \mbox{linearity}\\
= & \tr(\omega_1 \omega_2) & \mbox{}\\
\end{array}$$
\end{proof}

\chapter{The Construction}
\textit{In this chapter, we construct a closed (but unbounded) polyhedral complex with a morphism to the tropical moduli space of marked curves and show that the degree of this morphism captures information about the Hurwitz numbers.}

\section{Constructing the Polyhedral Complex, $\H(\overline{\sigma})$}
\textit{In this section, we specify a union of closed cones in real vector spaces and give modular meaning to points in these cones related to covers of marked tropical curves. Then we identify the points on the cones for which the modular meanings agree. Finally, we assign weights to the top-dimensional polyhedral cells.}

\begin{dfn}\label{fulllabel}
Fix $n \geq 3$ in $\N$, and $d \in \N$. Let $G$ be a topological type (\ref{topotype}) of trivalent trees with $n$ leaves and $m = n-3$ bounded edges (\ref{n-3}). Label the unbounded edges of $G$ by distinct elements of the set $\overline{\lambda} = \{\lambda_1,\ldots,\lambda_n\}$. For each edge, $\lambda_i$, pick an integer partition $\sigma_i$ of $d$; call this collection of choices $\overline{\sigma}$. Label the bounded edges of $G$ by distinct elements of the set $\overline{E} = \{e_1, \ldots, e_m\}$. For each edge, $e_i$, pick an integer partition $\nu_i$ of $d$; call this collection of choices $\overline{\nu}$. Let $G(\overline{\lambda},\overline{E})$ represent only the choice of the labelings on $G$. Let $G(\overline{\lambda},\overline{\sigma},\overline{E},\overline{\nu})$ represent this collection of choices, both the labelings and the associated integer partitions of $d$.
\end{dfn}

\begin{note}
Under the hypotheses in \ref{fulllabel}, each unbounded edge of $G$ contains a single leaf. As a result, we can also think of $\overline{\lambda}$ as a labeling of the elements of $L(G)$, the leaves of $G$.
\end{note}

The next definition will show how to interpret $G(\overline{\lambda},\overline{E})$ as a function from $\R^m_{\geq0}$ (the closure of the positive orthant in $\R^m$) to the set of metric graphs (with labeled leaves). In other words, this function gives modular meaning to the points in $\R^m_{\geq0}$.

\begin{dfn}
For any labelled topological type of trivalent trees, $G(\overline{\lambda},\overline{E})$, and point $p = (p_1,\ldots,p_m) \in \R^m_{\geq0}$, construct the following metric graph.
\begin{itemize}
\item If there are any $p_i = 0$, then degenerate $G$ by $e_i$ in the sense of \ref{degeneration}.
\item For any $p_i > 0$, assign length $p_i$ to edge $e_i$.
\end{itemize}
Notice that $\overline{\lambda}$ descends to a labeling of the unbounded edges (and leaves) of this new graph; notice also that the elements of $\overline{E}$ that were not degenerated also descend to a labeling of the bounded edges of this new graph. Call this labelled metric graph $G(\overline{\lambda},\overline{E})(p)$.
\end{dfn}

\begin{remark}
Notice that degeneration will never change the genus of a tropical curve.
\end{remark}

\begin{notation}
We will sometimes suppress the labelings of the edges in the notation $G(\overline{\lambda},\overline{\sigma},\overline{E},\overline{\nu})$ because they are universal, writing instead $G(\overline{\sigma},\overline{\nu})$.
\end{notation}

\begin{dfn}\label{acceptable}
Consider a set of choices $G(\overline{\sigma},\overline{\nu})$ on a trivalent tree $G$. Let $v$ be an internal (degree $3$) vertex in $G$. Thus $v$ is the endpoint of three edges, $\{\epsilon_1,\epsilon_2,\epsilon_3\}$. Each of these edges has an associated integer partition of $d$ as either a bounded or unbounded edge; call these the three integer partitions $\{\mu_1,\mu_2,\mu_3\}$. Recall from \ref{basisdfn} that each integer partition $\mu_i$ of $d$ corresponds to a basis element $K_{\mu_i}$ in the class algebra. Define
$$I(v) = \tr(K_{\mu_1}K_{\mu_2}K_{\mu_3}).$$
We will say that the vertex $v$ is \textbf{acceptable} if $I(v) \neq 0$. We say that $G(\overline{\sigma},\overline{\nu})$ is \textbf{acceptable} if every internal vertex $v$ in $G$ is acceptable.
\end{dfn}

\begin{remark}
Recall that each vertex in a tropical curve corresponds to an irreducible component in a classical marked stable curve. The acceptable condition corresponds to whether this data can be realized as a classical cover of $\p^1$ with only three ramification points.
\end{remark}

\begin{dfn}\label{Dsig}
Recall that we have fixed $n,d \in \N$, and let $m = n-3$. Pick a collection of $n$ integer partitions of $d$: $\overline{\sigma} = \{\sigma_1,\ldots,\sigma_n\}$. Define $D(\overline{\sigma})$ as the disjoint union of copies of $\R^m_{\geq0}$ indexed by the possible choices for acceptable $G(\overline{\sigma},\overline{\nu})$:
$$D(\overline{\sigma}) = \coprod_{\begin{array}{c}G(\overline{\sigma},\overline{\nu})\\ \mbox{acceptable}\end{array}} \R^m_{\geq0}.$$
\end{dfn}

\begin{remark}
The data for each cone has much in common; only the topological type and $\overline{\nu}$ are allowed to vary.
\end{remark}

\begin{dfn}
Pick a collection of $n$ integer partitions of $d$: $\overline{\sigma} = \{\sigma_1,\ldots,\sigma_n\}$. Let $\H(\overline{\sigma})$ be the topological space formed from $D(\overline{\sigma})$ by identifying points $p \in D(\overline{\sigma})_{G(\overline{\sigma},\overline{\nu})}$ and $q \in D(\overline{\sigma})_{G'(\overline{\sigma},\overline{\nu}')}$ exactly when there is an isomorphism (\ref{graphisodfn}) of metric graphs $F: G(\overline{\lambda},\overline{E})(p) \to G'(\overline{\lambda},\overline{E})(q)$ such that
\begin{itemize}
\item $F_E(\lambda_i) = \lambda_i$, and
\item if $F_E(e_i) = e_j$, then $\nu_i = \nu_j'$. 
\end{itemize}
Note that the first property says that the images of the leaves are specified, which is sufficient to specify the image of all vertices and edges in a tree. As a result, if such an isomorphism exists, it is unique. In addition, because $\overline{\sigma}$ is constant in $D(\overline{\sigma})$, the first property implies that the integer partition, $\sigma_i$, associated with each unbounded edge also matches when $p$ and $q$ are glued.
\end{dfn}

The space $\H(\overline{\sigma})$ is our candidate for the tropical analogue for the Hurwitz space.

\begin{thm}
The space $\H(\overline{\sigma})$ is a polyhedral complex in the sense of definition \ref{polycomplex}.
\end{thm}

\begin{proof}
Each cone $D(\overline{\sigma})_{G(\overline{\sigma},\overline{\nu})}$ is a copy of $\R^m_{\geq0}$, which is the closure of the positive orthant in $\R^m$ with the standard Euclidean topology. The space $D(\overline{\sigma})$ is a disjoint union of these topological spaces, which gives it a natural topology. The space $\H(\overline{\sigma})$ is formed by identifying points in $D(\overline{\sigma})$, which gives it a natural topology.

First we show that each copy of $\R^m_{\geq0}$ is homeomorphic to its image in $\H(\overline{\sigma})$. Let $p,q \in D(\overline{\sigma})_{G(\overline{\sigma},\overline{\nu})}$. If $p$ and $q$ are identified in $\H(\overline{\sigma})$, then there is an isomorphism of metric graphs $F: G(\overline{\lambda},\overline{E})(p) \to G(\overline{\lambda},\overline{E})(q)$ that respects the labelings. But because they come from the same copy of $D(\overline{\sigma})$, the edges that are degenerated must be degenerated for both $p$ and $q$. Thus the coordinates with value $0$ for $p$ and $q$ must agree. Since $F$ is an isomorphism of metric graphs, all of the non-degenerated edges must be the same length, so the positive coordinates of $p$ and $q$ must also be identical, meaning that $p=q$. In other words, each the projection $c_{D(\overline{\sigma})_{G(\overline{\sigma},\overline{\nu})}}: D(\overline{\sigma})_{G(\overline{\sigma},\overline{\nu})} \to \H(\overline{\sigma})$ is an injection. As a result, $c_{D(\overline{\sigma})_{G(\overline{\sigma},\overline{\nu})}}$ is a homeomorphism of a closed lattice polyhedron in $\R^m$ with its image in $\H(\overline{\sigma})$. 

Let $Y$ be a face of $D(\overline{\sigma})_{G(\overline{\sigma},\overline{\nu})}$; the preceding paragraph also shows that $c_{D(\overline{\sigma})_{G(\overline{\sigma},\overline{\nu})}}|_Y$ is a homeomorphism of the face $Y$ with its image. Given such a face, $Y$, let $c_{\hat{Y}}$ be the associated function and $\hat{Y}$ the image of that function.

Consider the sets
$$\overline{P} = \bigcup_{G,\overline{\nu}} \{\hat{Y} \mid Y \mbox{ is a face of } D(\overline{\sigma})_{G(\overline{\sigma},\overline{\nu})}\}$$
and
$$\overline{c} = \bigcup_{G,\overline{\nu}} \{c_{\hat{Y}} \mid Y \mbox{ is a face of } D(\overline{\sigma})_{G(\overline{\sigma},\overline{\nu})}\}.$$
We will now show that $\overline{P}$ and $\overline{c}$ give $\H(\overline{\sigma})$ the structure of a polehedral complex.

Notice that $\H(\overline{\sigma})$ is, by definition, the union of the images of the unrestricted charts, so it is certainly the union of these polyhedral cells.

Now we must show that the intersection of two polyhedral cells is a polyhedral cell. Consider two polyhedral cells $\hat{Y}, \hat{Y}' \in \overline{P}$. Then there is a trivalent graph $G$ and collection of integer partitions $\overline{\nu}$ such that $\hat{Y}$ is the image of a face $Y$ of $D(\overline{\sigma})_{G(\overline{\sigma},\overline{\nu})}$ and there is a trivalent graph $G'$ and collection of integer partitions $\overline{\nu}'$ such that $\hat{Y}'$ is the image of a face $Y'$ of $D(\overline{\sigma})_{G'(\overline{\sigma},\overline{\nu}')}$. 

Notice that the set of leaves in $G$ and $G'$ have the same labels. Let $D_1$ be the set of bounded edges, $e$, of $G$ such that there does not exist an edge $e'$ of $G'$ with $\pi(e') = \pi(e)$ (\ref{partitionpi}). Let $D_1'$ be the set of bounded edges, $e'$, of $G'$ such that there does not exist an edge $e$ of $G$ with $\pi(e') = \pi(e)$.

Let $H_1$ be the degeneration of $G$ by the edges of $D_1$; let $H_1'$ be the degeneration of $G'$ by the edges of $D_1'$. Because $G$ and $G'$ were trivalent trees, $H_1$ and $H_1'$ are stable trees (\ref{stabledegen}). As above, $\overline{\lambda}$ gives a bijection of $L(H_1)$ to $L(H_1')$; moreover, by construction, the set of (set) partitions created by bounded edges in each of $H_1$ and $H_1'$ are identical. So by \ref{graphisomorph}, there is an isomorphism $F: H_1 \to H_1'$ that agrees with the labeling on the vertices.

Let $D_2$ be the set of edges $e_i$ of $H_1$ such that $F_E(e_i) = e_j'$ but $\nu_i \neq \nu_j$. Let $D_2'$ be the set of edges $e_j'$ of $H_1'$ such that $F_E(e_i) = e_j'$ but $\nu_i \neq \nu_j$. Let $H_2$ be the degeneration of $H_1$ by the edges of $D_2$, and let $H_2'$ be the degeneration of $H_1'$ by the edges of $D_2'$. Clearly, $F$ descends to an isomorphism $F: H_2 \to H_2'$, and now $F$ respects all labelings and integer partitions.

Consider the face of $D(\overline{\sigma})_{G(\overline{\sigma},\overline{\nu})}$ determined by setting the coordinates associated with $D_1 \cup D_2$ to zero and intersecting with $Y$; call this face $Q$. Similarly, consider the face of $D(\overline{\sigma})_{G'(\overline{\sigma},\overline{\nu}')}$ determined by setting the coordinates associated with $D_1' \cup D_2'$ to zero and intersecting with $Y'$; call this face $Q'$. Pick a point $p \in Q$. Consider the following point $q$. If $p_i > 0$ and $F_E(e_i) = e_j$, then set $q_j = p_i$. Otherwise, let $q_j = 0$. By the above discussion, $G(\overline{\lambda},\overline{E})(p)$ is isomorphic to $G(\overline{\lambda},\overline{E})(q)$ by $F$. So the image of $p$ is in the intersection, $\hat{Y} \cap \hat{Y}'$. Similarly, the image of every point from $\hat{Q}'$ is in the intersection. So $\hat{Q} = \hat{Q}'$ is contained in $\hat{Y} \cap \hat{Y}'$.

Let $\hat{p} \in \hat{Y} \cap \hat{Y}'$; then $p$ represents a metric graph in $Y$ and $Y'$. The partitions of the leaves that can be realized in each version of $p$ must be a subset of those possible for each of $G$ and $G'$, so certainly $p$ lies in the set of points for which the edges of $D_1$ have been degenerated. Similarly, the two versions of $p$ must have identical integer partitions on the bounded edges, so $p$ also lives in the face on which the edges of $D_2$ have been degenerated. Hence $p$ lies in $Y$, meaning $\hat{p} \in \hat{Q}$. Thus $\hat{Q} = \hat{Q}'$ is the intersection, which is clearly a polyhedral cell.

Finally, we must show that the different charts for the faces are lattice equivalent. Suppose $\hat{Y} = \hat{Y'}$, where $Y$ is a face of $D(\overline{\sigma})_{G(\overline{\sigma},\overline{\nu})}$ and $Y'$ is a face of $D(\overline{\sigma})_{G'(\overline{\sigma},\overline{\nu}')}$. Then for any point, $p \in Y$, there is a point $p' \in \hat{Y}'$ with which it is identified. The coordinates of $p$ and $p'$ corresponding to the sets $D_1 \cup D_2$ constructed above are all zero. The other coordinates are permuted by the isomorphism, and this permutation is the same for all points. Thus the affine spaces containing $Y$ and $Y'$ are identified globally by a permutation of the coordinates. This isomorphism clearly respects the lattices, so these charts are lattice-equivalent.

Having checked all conditions in \ref{polycomplex}, we see that $\H(\overline{\sigma})$ is a polyhedral complex.
\end{proof}

\begin{thm}
The space $\H(\overline{\sigma})$ is connected.
\end{thm}

\begin{proof}
The graphs $G(\overline{\sigma},\overline{\nu})(0,\ldots,0)$ and $G'(\overline{\sigma},\overline{\nu}')(0,\ldots,0)$ are both trees with a degree $n$ vertex at the end of $n$ unbounded edges. Clearly, there exists an isomorphism of these trees respecting $\overline{\lambda}$ (and thus $\overline{\sigma}$), and there are no bounded edges left to consider, so that isomorphism respects $\overline{E}$ and $\overline{\nu}$. So these two points are identified in $\H(\overline{\sigma})$. The point $(0,\ldots,0)$ is in every face of these cones as well. So, the image point is in the image of every chart (meaning every polyhedral cell). In the Euclidean topology, that is enough for us to know that the space is connected.
\end{proof}

\begin{remark}
Although $\H(\overline{\sigma})$ is not embedded, the previous theorem shows that it has a fan-like structure.
\end{remark}

\begin{dfn}\label{weightsdfn}
Consider a top-dimensional polyhedron $P_i \in \overline{P}$, meaning that it is the image of the full copy of $\R^m_{\geq0}$ from $D(\overline{\sigma})_{G(\overline{\lambda},\overline{\sigma},\overline{E},\overline{\nu})}$, where $G = (V,E)$ is a trivalent tree. Define the weight of $P_i$ as
$$w(P_i) = \frac{1}{d!} \left(\prod_{e_i \in \overline{E}} \frac{1}{|K_{\nu_i}|}\right) \left(\prod_{v \in (V \setminus L(G))} I(v)\right).$$
Notice that the condition $v \in V \setminus L(G)$ is the same as saying that $v$ is an internal vertex.
\end{dfn}

\section{The Morphism to $\M_{0,n}$}
\textit{In this section, we define a morphism from the polyhedral complex constructed above to Mikhalkin's moduli space of marked tropical rational curves and show that the degree of that morphism captures information about the Hurwitz numbers.}

Many of the ideas in this section are derived from \cite{MikM}, including the statements of \ref{doubleratio} and \ref{ratiocoord}. More detail is provided here because it is absent from the original.

\begin{dfn}\label{doubleratio}\cite{MikM}
Consider one of the polyhedra from \ref{Dsig}, $D(\overline{\sigma})_{G(\overline{\lambda},\overline{\sigma},\overline{E},\overline{\nu})}$. Fix $4$ distinct elements, $w,x,y,z \in \overline{\lambda}$, and think of these as labels of the leaves. Define the function $d_{(w,x),(y,z)}: D(\overline{\sigma})_{G(\overline{\lambda},\overline{\sigma},\overline{E},\overline{\nu})} \to \R$ as follows. First notice that, because $G$ is a tree, there is a unique (oriented) path, $P_{(w,x)}$, in $G$ from $w$ to $x$. Similarly, there is a unique (oriented) path, $P_{(y,z)}$, in $G$ from $y$ to $z$. Call this intersection $P_{(w,x),(y,z)}$. Notice that $P_{(w,x),(y,z)}$ is connected; if the paths separate and rejoin, that would give a circuit. Notice that no unbounded edge can be used in more than one of these paths because the leaves are distinct (and paths don't repeat edges); so, $P_{(w,x),(y,z)}$ is contained in the bounded edges of $G$. Let $p \in D(\overline{\sigma})_{G(\overline{\lambda},\overline{\sigma},\overline{E},\overline{\nu})}$; then $G(\overline{\lambda},\overline{E})(p)$ contains a copy of the intersection, $P_{(w,x),(y,z)}(p)$, but here each edge has a length. If the orientations of $P_{(w,x)}$ and $P_{(y,z)}$ agree on $P_{(w,x),(y,z)}(p)$,  then let $d_{(w,x),(y,z)}(p)$ be the length of $P_{(w,x),(y,z)}(p)$. If the orientations disagree, let $d_{(w,x),(y,z)}(p)$ be the negative of the length of $P_{(w,x),(y,z)}(p)$. We call $d_{(w,x),(y,z)}$ the \textbf{double ratio} associated with these distinct pairs of ordered points.
\end{dfn}

\begin{lemma}
On each face $Y$ of $D(\overline{\sigma})_{G(\overline{\sigma},\overline{\nu})}$, $d_{(w,x),(y,z)} \circ c_{\hat{Y}}$ is linear.
\end{lemma}

\begin{proof}
Pick $p,q \in D(\overline{\sigma})_{G(\overline{\sigma},\overline{\nu})}$ and $r \in \R^+$. Notice that the choice of a double ratio determines all orientation issues, so we do not need to worry about the signs when working with a single function. The length of $P_{(w,x),(y,z)}(rp)$ is computed from $P_{(w,x),(y,z)}(p)$ by first scaling the length of each included edge by $r$ and then summing, which gives the same result as computing the length of $P_{(w,x),(y,z)}(p)$ and then scaling the total length by $r$. Similarly, the length of $P_{(w,x),(y,z)}(p+q)$ is computed by first adding the edge lengths for $p$ and $q$ and then computing the intersection length, which gives the same result as computing the lengths of $P_{(w,x),(y,z)}(p)$ and $P_{(w,x),(y,z)}(q)$ and then adding them.
\end{proof}

\begin{lemma}\label{ratiocoords}\cite{MikM}
Let $p \in D(\overline{\sigma})_{G(\overline{\sigma},\overline{\nu})}$. Then, for each $i$, there exists a choice of $w,x,y,z \in \overline{\lambda}$ such that $|d_{(w,x),(y,z)}(p)| = p_i$. In other words, each edge length is a double ratio.
\end{lemma}

\begin{proof}
Fix $1 \leq i \leq n$ and consider $e_i = \{v,w\}$. If you are familiar with the classical or tropical moduli spaces of marked stable curves, you could find this as the pull-back of the only double ratio on $\M_{0,4}$ by the forgetful morphism $\M_{0,n} \to \M_{0,4}$ defined by those $4$ marked points. But we can locate this double ratio directly as follows.

We will construct two paths whose intersection is $e_i$. Because $e_i$ is bounded and $G$ is trivalent, there are two other edges coming out of each of $v$ and $w$: $\{v,a\}, \{v,b\}, \{w,c\}, \{w,d\}$. Remove $\{v,a\}$ from $G$ and consider $\pi(\{v,a\}) = \{S(\{v,a\}),S'(\{v,a\})\}$. Let $w$ be any leaf coming from the part of the partition associated with the component that contains $a$. Similarly, pick $x,y,z$, being careful to use either $c$ or $d$ to produce $x$. Then the intersection of these paths is clearly $e_i$, so $d_{(w,x),(y,z)}(p) = \pm p_i$. The same argument holds for degenerations of $G$ by simply picking $2$ pairs of adjacent vertices to play the roles of $a,b,c,d$. In the degeneration, there may be multiple ways to realize the edge's length as a double ratio.
\end{proof}

\begin{remark}
The proof above shows that $e_i$ is part of $P_{(w,x),(y,z)}$ if $\pi(e_i)$ separates $w$ from $x$ and $y $ from $z$. We will say that $d_{(w,x),(y,z)}$ is \textbf{$e_i$-compatible} in this case.
\end{remark}

\begin{lemma}\label{ratiocoord}
The length of $e_i$ for a point $p$ is the minimal, non-zero (absolute) value of the $e_i$-compatible double ratios evaluated at $p$.
\end{lemma}

\begin{proof}
By \ref{ratiocoords}, the length of $e_i$ does appear on the list of values, up to a sign change. Notice, however, that the lengths of the segments are all positive. So any path containing $e_i$ and other edges must be strictly longer than the path containing only $e_i$ and hence have a larger absolute value.
\end{proof}

\begin{remark}
The definition of $d_{(w,x),(y,z)}$ clearly agrees on all copies of points from different cones that are identified in $\H(\overline{\sigma})$ because it only depends on the metric structure on the graph. So we can think of $d_{(w,x),(y,z)}$ as a function from $\H(\overline{\sigma})$ to $\R$.
\end{remark}

\begin{remark}
Notice that $d_{(x,w),(y,z)} = -d_{(w,x),(y,z)} = -d_{(x,w),(z,y)}$. Also notice that $d_{(y,z),(w,x)} = d_{(w,x),(y,z)}$. We say that two double ratios are equivalent if they differ only by these kinds of reorderings (but respect the vertex pairs). The equivalence classes depend only on a choice of $4$ vertices and a way of putting those vertices into disjoint pairs. Hence there are $N = 3\left(\begin{array}{c}n\\4\end{array}\right)$ equivalence classes.
\end{remark}

\begin{dfn}
Pick $N = 3\left(\begin{array}{c}n\\4\end{array}\right)$ double ratios, one from each equivalence class, and order them $(d_1, \ldots d_N)$. Define $\phi: \H(\overline{\sigma}) \to \R^N$ by $p \mapsto (d_1(p),\ldots,d_N(p))$.
\end{dfn}

\begin{cor}\label{latticeindex}
Let $q \in \phi(\H(\overline{\sigma}))$ such that the coordinates of $q$ are integral. Then for any $p \in \H(\overline{\sigma})$ such that $q = \phi(p)$, the coordinates of $p$ (in any chart) are integral.
\end{cor}

\begin{proof}
Being integral in the image means every coordinate is integral, including the ones that measure the edge lengths. So any preimage point has edge lengths that are all integral.
\end{proof}

\begin{cor}
For each face $Y$ of $D(\overline{\sigma})_{G(\overline{\sigma},\overline{\nu})}$, $\phi \circ c_{\hat{Y}}$ is a linear isomorphism onto its image.
\end{cor}

\begin{proof}
Linearity follows from the linearity of the coordinate functions. Injectivity follows from the fact that the coordinates from the domain are coordinate functions.
\end{proof}

\begin{lemma}
The image of $\phi: \H(\overline{\sigma}) \to \R^N$ is independent of $\overline{\sigma}$.
\end{lemma}

\begin{proof}
The definition of $\phi$ used only the metric graph structure and did not mention the integer partitions of $d$ associated with each unbounded edge.
\end{proof}

\begin{remark}
Different choices of double ratios to define the map $\phi$ will change the image, but as Mikhalkin points out, the image differs only by negating some coordinates, which clearly produces an isomorphic polyhedral structure.
\end{remark}

\begin{notation}
Let $\tilde{M} = \phi(\H(\overline{\sigma}))$.
\end{notation}

\begin{thm}
The collections $\phi(\overline{P}) = \{\phi(P_i) | P_i \in \overline{P}\}$ and $\phi(\overline{c}) = \{\phi \circ c_i | c_i \in \overline{c}\}$ give $\tilde{M}$ the structure of a polyhedral complex.
\end{thm}

\begin{proof}
The space $\tilde{M}$ is a subset of $\R^N$, which gives it a topology. Because the maps are linear, the images of the closed polyhedra are still closed in $\R^N$. In addition, because the maps are linear isomorphisms, they are also homoemorphisms on the cones. The space $\tilde{M}$ is the image of the union of the polyhedra $\overline{P}$, but this is trivially the union of the images, which are the polyhedra in $\phi(\overline{P})$. 

As before, showing that the intersection of two polyhedra is a polyhedron is a more subtle than the rest of this proof. Note that $\phi$ forgets the integer partitions, $\overline{\nu}$. Given two polyhedra in $\phi(\H(\overline{\sigma}))$, there are many choices of preimage polyhedra. As long as we pick the preimages with the same choice of $\overline{\nu}$, then the intersection of the lifts will have as its image the intersection of the images.

We already know that the charts in $\overline{c}$ are lattice-equivalent on intersections. In $\phi(\overline{c})$, each of these is post-composed with $\phi$, the same linear isomorphism, which clearly preserves lattice-equivalence. So, $\tilde{M}$ is a polyhedral complex.
\end{proof}

\begin{lemma}
Given the polyhedral structures above, $\phi: \H(\overline{\sigma}) \to \tilde{M}$ is a morphism of polyhedral complexes.
\end{lemma}

\begin{proof}
Be the very definition of $\tilde{M}$, each polyhedral cell from $\H(\overline{\sigma})$ maps into a polyhedral cell of $\tilde{M}$. The check that $\phi$ is locally affine and integral is trivial: substituting in the definition of the charts on $\tilde{M}$, we see that $\phi$ is composed with its inverse. This cancels, leaving the original maps from the lattice equivalence in $\H(\overline{\sigma})$, for which the desired property has already been shown.
\end{proof}

\begin{thm}
Given the weights defined in \ref{weightsdfn}, $\deg(\phi) = \frac{1}{d!}\tr(K_{\sigma_1} \cdots K_{\sigma_n})$.
\end{thm}

\begin{proof}
We will prove this theorem by induction on the number of internal vertices in the topological types, $k = n-2$. But first we simplify the computation in all cases.

Pick $q$ in the interior of a top-dimensional polyhedral cell in $\tilde{M}$. All of the polyhedral cells map isomorphically through $\phi$, so every preimage of $q$ is in the interior of top-dimension polyhedral cell in $\H(\overline{\sigma})$. By \ref{latticeindex}, the lattice from each of these domain points maps onto the entire lattice in the codomain. So for any preimage $p$, $\ind(\phi)_p = 1$. Hence we must compute
$$\deg(\phi)_q = \sum_{\begin{array}{c}p \in \H(\overline{\sigma})\\\phi(p) = q\end{array}} w(p) \cdot \ind(\phi)_p = \sum_{\begin{array}{c}p \in \H(\overline{\sigma})\\\phi(p) = q\end{array}} w(p).$$

Recall that, for $p$ in the interior of $D(\overline{\sigma})_{G(\overline{\lambda},\overline{\sigma},\overline{E},\overline{\nu})}$, $$w(p) = \frac{1}{d!} \left(\prod_{e_i \in \overline{E}} \frac{1}{|K_{\nu_i}|}\right) \left(\prod_{v \in (V \setminus L(G))} I(v)\right).$$

First, suppose $k=1$, the smallest possible number of internal vertices in a trivalent graph. In the unique (topological type of a) trivalent graph with only one internal vertex, there are no bounded edges. In addition, there is only one vertex of degree $3$. If $v$ is that degree $3$ vertex, then $I(v) = \tr(K_{\sigma_1} K_{\sigma_2} K_{\sigma_3})$. So for any $p \in \phi^{-1}(q)$,
$$w(p) = \frac{1}{d!} \cdot 1 \cdot \tr(K_{\sigma_1} K_{\sigma_2} K_{\sigma_3}).$$
Notice also that there are no choices for $\overline{\nu}$ in a graph without bounded edges. So there is only this one preimage point $p = \phi^{-1}(q)$, and this single index is actually the the degree.

Now suppose that $k > 1$ and that the expression is known for all trivalent trees with $j < k$ internal vertices. The topological type of the preimages of $q$ can be determined from the coordinates of $q$, as seen in \cite{MikM}. 
Call this topological type $G = (V,E)$. By \ref{twoleaftripod}, $G$ has an internal vertex that is the intersection of two unbounded edges. Permute the labeling $\overline{\lambda}$ (simultaneously on all of $\H(\overline{\sigma})$) so that the unbounded edges $\lambda_{n-1}$ and $\lambda_n$ intersect at the internal vertex $\hat{v}$. In addition, permute the labeling $\overline{E}$ such that the third edge at $\hat{v}$ is $e_m$.

Consider the following topological type of graphs, $\hat{G} = (\hat{V},\hat{E})$, defined as follows. Thinking of $\overline{\lambda}$ as a labeling of $L(G)$, let $\hat{V} = V \setminus \{\lambda_{n-1}, \lambda_n\}$. Thinking of $\overline{\lambda}$ as a labeling of the unbounded vertices, let $\hat{E} = E \setminus \{\lambda_{n-1}, \lambda_n\}$. In short, $\hat{G}$ is formed from $G$ be removing two unbounded edges that intersect (and their leaves).
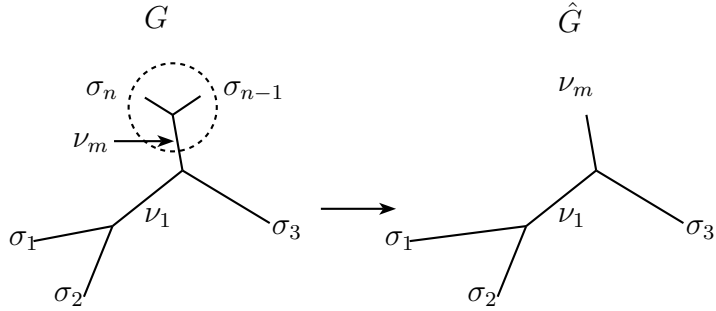
\begin{center}
\psset{xunit=0.5cm,yunit=0.5cm,algebraic=true,dotstyle=o,dotsize=3pt 0,linewidth=0.8pt,arrowsize=3pt 2,arrowinset=0.25}
\begin{pspicture*}(-4,-1.1)(16,7.7)
\psline(-2.94,0.94)(-0.84,1.34)
\psline(-1.6,-0.54)(-0.84,1.34)
\psline(-0.84,1.34)(1.02,2.82)
\psline(1.02,2.82)(3.34,1.42)
\psline(1.02,2.82)(0.76,4.3)
\psline(0.76,4.3)(1.5,4.8)
\psline(0.76,4.3)(0.02,4.76)
\pscircle[linestyle=dashed,dash=2pt 2pt](0.76,4.5){0.6}
\rput[tl](2.1,5.2){$\sigma_{n-1}$}
\rput[tl](-1.55,5.2){$\sigma_n$}
\rput[tl](-1.9,3.8){$\nu_m$}
\rput[tl](-3.6,1.2){$\sigma_1$}
\rput[tl](-2.4,-0.3){$\sigma_2$}
\rput[tl](3.4,1.5){$\sigma_3$}
\rput[tl](0,1.8){$\nu_1$}
\psline{->}(-0.8,3.6)(0.8,3.6)
\rput[tl](0,7.2){$G$}

\psline{->}(4.7,1.8)(6.7,1.8)

\psline(7.06,0.94)(10.16,1.34)
\psline(9.4,-0.54)(10.16,1.34)
\psline(10.16,1.34)(12.02,2.82)
\psline(12.02,2.82)(14.34,1.42)
\psline(12.02,2.82)(11.76,4.3)
\rput[tl](11,5.3){$\nu_m$}
\rput[tl](6.4,1.2){$\sigma_1$}
\rput[tl](8.6,-0.3){$\sigma_2$}
\rput[tl](14.4,1.5){$\sigma_3$}
\rput[tl](11,1.8){$\nu_1$}
\rput[tl](11,7.3){$\hat{G}$}
\end{pspicture*}
\end{center}
Notice that $\hat{v}$, which was internal in $G$, is now a leaf. So $\overline{\lambda}_{\hat{G}} = \{\lambda_1,\ldots,\lambda_{n-2},e_m\}$ and $\overline{E}_{\hat{G}} = \{e_1,\ldots,e_{m-1}\}$ are labelings of the unbounded and bounded edges of $\hat{G}$ respectively. Moreover, $\overline{\sigma}_{\hat{G}} = \{\sigma_1,\ldots,\sigma_{n-2},\nu_m\}$ and $\overline{\nu}_{\hat{G}} = \{\nu_1,\ldots,\nu_{m-1}\}$ are collections of integer partitions associated to these edges. This information determines a polyhedral cell in a version of $\H(\overline{\sigma})$, $\hat{H}$, with one fewer internal vertices.

Each polyhedral cell containing a preimage of $q$ has a distinct image in $\hat{H}$ obtained by simply forgetting the length of $e_m$. Similarly, $q$ has an analogous point in the image of these points. 
So 
$$\deg(\phi)_q = \sum_{\begin{array}{c}p \in \H(\overline{\sigma})\\\phi(p) = q\end{array}} w(p) = \sum_{\begin{array}{c}p \in \H(\overline{\sigma})\\\phi(p) = q\end{array}} \frac{1}{d!}\prod_{e_i \in \overline{E}} \frac{1}{|K_{\nu_i}|} \prod_{v \in (V \setminus L(G))} I(v)$$
Recall that the data of a preimage point is the same as a choice of $\overline{\nu}$. Also, we can factor out the parts of the product coming from the vertex $\hat{v}$.
$$= \sum_{\overline{\nu}} \left(\frac{1}{|K_{\nu_m}|} I(\hat{v})\right) \left(\frac{1}{d!} \prod_{e_i \in \overline{E}_{\hat{G}}} \frac{1}{|K_{\nu_i}|} \prod_{v \in (\hat{V} \setminus L(\hat{G}))} I(v)\right)$$
The sum over $\overline{\nu}$ can be decomposed into a sum over each term in $\overline{\nu}$. Notice that the two factored terms only depend on $\nu_m$, so we can bring it outside that part of the sum.
$$= \sum_{\nu_m} \left(\frac{1}{|K_{\nu_m}|} I(\hat{v})\right) \sum_{\overline{\nu}_{\hat{G}}} \left(\frac{1}{d!} \prod_{e_i \in \overline{E}_{\hat{G}}} \frac{1}{|K_{\nu_i}|} \prod_{v \in (\hat{V} \setminus L(\hat{G}))} I(v)\right)$$
By our inductive hypothesis, the internal sum is just the degree of the morphism $\hat{\phi}: \hat{H} \to \hat{M}$, so we may substitute.
$$\deg(\phi)_q = \sum_{\nu_m} \left(\frac{1}{|K_{\nu_m}|} \tr(K_{\sigma_{n-1}} K_{\sigma_n} K_{\nu_m})\right) \left(\frac{1}{d!} \tr(K_{\sigma_1 }\cdots K_{\sigma_{n-2}} K_{\nu_m})\right)$$
Not every possible such product appears, but the ones that have been removed have value zero (see \ref{acceptable}), so we may assume they are present in this sum as well. Then the orthogonal basis lemma, \ref{states}, allows us to simplify to
$$\deg(\phi)_q = \frac{1}{d!} \tr(K_{\sigma_1} \cdots K_{\sigma_n}).$$
Notice that this expression does not depend on the polyhedron containing $q$, so
$$\deg(\phi) = \frac{1}{d!}\tr(K_{\sigma_1} \cdots K_{\sigma_n}).$$
\end{proof}

\begin{remark}
If the last proof is hard to conceptualize, think about it in a slightly different way. Unpacking the definition of $I(v) = \tr(K_{\mu_1}K_{\mu_2}K_{\mu_3})$ in the sum in the previous proof, we see that every integer partition, $\nu_i$, will appear in two distinct trace functions in the product and each integer partition $\sigma_i$ will appear in one. Repeated applications of the orthogonal basis lemma will absorb every factor of $\frac{1}{|K_{\nu_i}|}$ and combine the products of traces into a single trace containing one copy of $K_{\sigma_i}$ for each integer partition in $\overline{\sigma}$. 
\end{remark}

\begin{remark}
The space $\tilde{M}$ is exactly Mikhalkin's moduli space of tropical, genus $0$ curves with $n$ marked points, $\M_{0,n}$. He uses open cells, so his polyhedra correspond to the relative interiors of my polyhedral cells. He also uses \textit{combinatorial types} of graphs, which correspond to labeling just the leaves of the trees; my function $G(\overline{\lambda},\overline{E})$ also labels the the bounded edges, but this labeling is just notation to talk about the integer partitions $\overline{\nu}$ in a consistent manner. We then both add lengths to the bounded edges. This means that the map $\phi$ factors through his embedding of $\M_{0,n}$ into $\R^m$.
\end{remark}

So, we summarize:

\begin{thm}
Given a ramification profile, $\overline{\sigma}$, there is a connected polyhedral complex, $\H(\overline{\sigma})$ and a morphism of polyhedral complexes $\phi: \H(\overline{\sigma}) \to \M_{0,n}$ such that $\deg(\phi) = h(\ol{\sigma})$. Moreover, there is a modular interpretation of points in $\H(\overline{\sigma})$ as tropical ramified covers so that $\phi$ is the forgetful morphism taking a cover to its codomain with marked points at the ramification values.
\end{thm}

\begin{remark}
Mikhalkin compactifies his space in \cite{MikM} by allowing the lengths of the bounded edges to grow to infinity and shows that the compact object is still smooth. We could do the analogous construction and extend our morphism to the compactified case. However, degree is defined for us only in the interiors of the top dimensional cones, so this adds nothing. In addition, the top-dimensional cones already correspond to the most degenerate classical curves, and further degenerating adds no new interesting curve from the modular perspective. However, mathematicians with a more combinatorial perspective on tropical geometry may wish to see Mikhalkin's discussion of the compactification in \cite{MikM}.
\end{remark}

\subsection{Discussion}

\begin{remark}
If we let $d = 1$ in the construction above, then we get a version of our Hurwitz space. However, notice that there would then be no choices for the labels of the edges, so $\phi$ would be an injection, and we would have recovered the construction of Mikhalkin's moduli space, $\M_{0,n}$.
\end{remark}

\begin{remark}
The condition of being balanced as a polyhedral complex is what we need to guarantee that there can be a consistent notion of intersection theory (including the notion of degree) in tropical geometry. While we have not embedded $\H(\overline{\sigma})$ in a large vectorspace, the result above indicates that there is not an obstruction to putting a tropical structure on $\H(\overline{\sigma})$ compatible with the structure that Mikhalkin gives to $\M_{0,n}$ (i.e. finding an embedding of $\H(\overline{\sigma})$ as a balanced polyhedral complex, which would make it an honest tropical object).
\end{remark}

\begin{questions}
There are several questions that need answers.
\begin{itemize}
\item Does $\H(\overline{\sigma})$ have an embedding as a (simply) balanced polyhedral complex?
\item Does the the stratification of $\H(\overline{\sigma})$ correspond to the stratification of the boundary of the classical Hurwitz space?
\item Is there an algorithm for transforming the data from our definition of a tropical ramified cover into something that looks more like an honest tropical admissible cover? Are those covers the tropicalization of a classical admissible covers?
\item Will this construction carry over into the case of higher genus? Recent work by Kozlov (\cite{Kozlov2,Kozlov}) and Caporaso (\cite{Caporaso2,Caporaso}) indicates that the moduli spaces of higher genus curves can be given tropical structures much like Mikhalkin's from genus zero. Instead of being like Real manifolds, these spaces are like oribifolds.
\end{itemize}
\end{questions}

\section{Extensions}
\textit{Here we realize that most of this construction works if the class algebra is replaced by a general Frobenius algebra.}

\begin{dfn}
Let $V$ be a finite dimensional, unital algebra over a field $k$. Then $V$ is a \textbf{Frobenius algebra} over $k$ if $V$ has a non-degenerate bilinear pairing $h : V \times V \to k$ such that, for any triple of elements $a,b,c \in V$, 
$$h(ab,c) = h(a,bc).$$
\end{dfn}

\begin{thm}
The class algebra, $Z(\R[S_d])$, is a Frobenius algebra.
\end{thm}

\begin{proof}
Define the bilinear pairing $h : V \times V \to \R$ by:
$$h(g,g') = \tr(gg').$$
Note that if $g$, $g'$, and $g''$ are in the class algebra, then $h(gg',g'') = \tr(gg'g'') = h(g,g'g'')$, so $V$ is a Frobenius algebra.
\end{proof}

There are only a handful of aspects of our construction that depended on the class algebra.
\begin{enumerate}
\item The bilinear pairing gives the trace function, but there is not an obvious orthogonal basis for the Frobenius algebra.
\item Instead of choosing integer partitions for each edge of a topological type of tropical graph, we would choose basis vectors for each edge.
\item We already realized that $|K_\sigma| = \tr(K_\sigma K_\sigma)$. We could replace this quantity in the expressions above with $h(v,v)$ for a basis vector $v$. It is not at all clear what role these numbers play. If we instead replace each basis vector above by $K_\sigma \to \frac{K_\sigma}{\sqrt{|K_\sigma|}}$ in order to make the basis orthonormal, then the weights become
$$w(P_i) = \frac{1}{d!} \left(\prod_{\lambda_i \in \overline{\lambda}} \sqrt{|K_{\sigma_i}|} \right) \left(\prod_{v \in (V \setminus L(G))} I(v)\right).$$
The first two terms in this product no longer depend on the cone at all, but it is not clear what either term would mean in another Frobenius algebra.
\end{enumerate}

There are a few connections that we can make at this time.

\begin{remark}
Notice that the induction in the main theorem that rips off an internal vertex is really a special case of the famous Cut-and-Join formula. Also, notice that the vertex $\hat{v}$ corresponds classically to a copy of $\p^1$ with three special points. If you consider these points to be punctures, this object is the famous ``pair of pants" from a $2$-dimensional topological quantum field theory. The category of 2D-TQFTs is known to be equivalent to the category of Frobenius algebras. It is, however, not clear if there is any reasonable classical geometry interpretation of this construction in general as there was in the case of admissible covers.
\end{remark}


\baselineskip=15.5pt plus .5pt minus .2pt







\printindex

\nocite{*}      

\bibliographystyle{plain}  
\bibliography{biblio}        
\index{Bibliography@\emph{Bibliography}}%


\begin{vita}
\index{Vita@\emph{Vita}}%
Brian Paul Katz
\index{Brian Paul Katz}%
was born in Greensboro, NC, USA on November 26, 1980, the son of Jefferey David Katz and Laurie Ann North Katz. After completing his High School studies in Greensboro, NC in 1999, he entered Williams College in Williamstown, MA. He received a degree of B.\,A. in Mathematics, Music, and Chemistry cum laude with Honors from Williams College in 2003. In the fall of 2003 he started graduate studies in the department of Mathematics at The University of Texas at Austin where he was employed as a graduate research assistant and teaching assistant. In the fall of 2006 he became an assistant instructor in the same department. In the spring of 2009 he accepted a position as an assistant professor in the Department of Mathematics and Computer Science at Augustana College in Rock Island, Illinois, where he now teaches.
\end{vita}

\end{document}